\documentclass[reqno,a4paper,oneside]{amsart}
\usepackage{hyperref}
\usepackage{amssymb}
\usepackage{colortbl}

\begin{document}

\title[Discontinuous Galerkin approximations]
{Discontinuous Galerkin approximations for an optimal control problem of three-dimensional Navier-Stokes-Voigt equations}

\author [C.T. Anh, T.M. Nguyet]
{Cung The Anh$^{\natural}$ and Tran Minh Nguyet}

\address{Cung The Anh \hfill\break
Department of Mathematics, Hanoi National University of Education \hfill\break
136 Xuan Thuy, Cau Giay, Hanoi, Vietnam}
\email{anhctmath@hnue.edu.vn} 
\address{Tran Minh Nguyet\hfill\break
Department of Mathematics, Thang Long University \hfill\break
Nghiem Xuan Yem, Hoang Mai, Hanoi, Vietnam}
\email{tmnguyettlu@gmail.com}


\begin{abstract} 
We analyze a fully discrete scheme based on the discontinuous (in time) Galerkin approach, which is combined with conforming finite element subspaces in space, for the distributed optimal control problem of the three-dimensional Navier-Stokes-Voigt equations with a quadratic objective functional and box control constraints. The space-time error estimates of order $O(\sqrt{\tau}+h)$, where $\tau$ and $h$ are respectively the time and space discretization parameters, are proved for the difference between the locally optimal controls and their discrete approximations.
\end{abstract}
\thanks{$^{\natural}$ Corresponding author: anhctmath@hnue.edu.vn}

\maketitle

\numberwithin{equation}{section}
\newtheorem{theorem}{Theorem}[section]
\newtheorem{remark}{Remark}[section]
\newtheorem{definition}{Definition}[section]
\newtheorem{lemma}{Lemma}[section]
\newtheorem{corollary}{Corollary}[section]
\newtheorem{proposition}[theorem]{Proposition}

\section{Introduction}
Let $\Omega$ be an open bounded domain in $\mathbb{R}^3$ with boundary $\partial\Omega$. We denote the space-time cylinder by $Q=\Omega\times (0,T)$, where $T>0$ is given. In this paper we prove some error estimates for the numerical approximation of a distributed optimal control problem governed by the evolution Navier-Stokes-Voigt equations with pointwise control constraints. More precisely, we consider the following problem
\begin{equation*}
(P)\hspace{0.5cm}
\begin{cases}
 \min\, J(y,u)\\
u\in U_{\alpha,\beta},
\end{cases}
\end{equation*}
where
\begin{multline*}
J(y,u)=\dfrac{\alpha_T}{2}\int_{\Omega}|y(x,T)-y_T(x)|^2dx + \dfrac{\alpha_Q}{2}\iint_{Q}|y(x,t)-y_Q(x,t)|^2dxdt\\
+\dfrac{\gamma}{2}\iint_{Q}|u(x,t)|^2dxdt,
\end{multline*}
and $U_{\alpha,\beta}$ is the set of admissible controls defined for given real constants $\alpha_j,\beta_j, j=1,2,3$, by
\begin{equation*}
\label{AdSet}
U_{\alpha,\beta}=\{u\in \mathbb{L}^2(Q): \alpha_j\le u_j(x,t)\le \beta_j \text{ a.e. } (x,t)\in Q, j=1,2,3\}.
\end{equation*}
Here the free variables - state $y$ and control $u$ - have to fulfill the following 3D Navier-Stokes-Voigt (sometimes written Voight) equations
\begin{equation}
\text{}
\begin{cases}\label{NSV1}
y_t- \nu\Delta y -\alpha^2\Delta y_t +(y\cdot\nabla)y+\nabla p&=u, \;x\in \Omega, t>0,\\
\hfil \nabla \cdot y&=0,\; x\in \Omega, t>0,\\
\hfil y(x,t)&=0, \; x\in\partial\Omega, t>0,\\
\hfil y(x,0)&=y_0(x),\; x\in\Omega.\\
\end{cases}
\end{equation}
In system \eqref{NSV1}, $y=y(x,t)=(y_1(x,t),y_2(x,t),y_3(x,t))$ is the velocity, $y_0=y_0(x)$ is the initial velocity, $p=p(x,t)$ is the pressure, $\nu >0$ is the kinematic viscosity coefficient, and $\alpha\ne 0$ is the length-scale parameter characterizing the elasticity of the fluid. 

To study the above optimal control problem we assume that 
\begin{itemize}
\item The domain $\Omega$ is an open bounded subset of $\mathbb{R}^3$ with $C^2$ boundary $\partial \Omega.$
\item The initial value $y_0$ is a given function in $D(A)$. The desired states have to satisfy $y_T\in V$ and $y_Q\in \mathbb{L}^2(Q).$
\item The coefficients $\alpha_T,\alpha_Q$ are non-negative real numbers, where at least one of them is positive to get a non-trivial objective functional. The regularization parameter $\gamma$, which measures the cost of the control, is also a positive number.
\end{itemize}

The Navier-Stokes-Voigt equations was introduced by Oskolkov in \cite{Oskolkov} as a model of motion of certain linear viscoelastic incompressible fluids. This  system was also proposed by Cao, Lunasin and Titi in \cite{Cao} as a regularization, for small values of $\alpha$, of the 3D Navier-Stokes equations for the sake of direct numerical simulations. The presence of the regularizing term $-\alpha^2 \Delta u_t$ in \eqref{NSV1} has two important consequences. First, it leads to the global well-posedness of \eqref{NSV1} both forward and backward in time, even in the case of three dimensions. Second, it  changes the parabolic character of the limit Navier-Stokes equations, so the Navier-Stokes-Voigt system behaves like a damped hyperbolic system. In fact, the Navier-Stokes-Voigt system is perhaps the newest model in the the so-called $\alpha$-models in fluid mechanics (see e.g. \cite{Holst2010}), but it has attractive advantage over other $\alpha$-models in that one does not need to impose any additional artificial boundary condition (besides the Dirichlet boundary conditions) to get the global well-posedness. We also refer the reader to \cite{EbrahimiHolstLunasin}  for an interesting  application of the Navier-Stokes-Voigt equations in image inpainting.

In the past years, the existence and long-time behavior of solutions to the Navier-Stokes-Voigt equations has attracted the attention of many mathematicians. In bounded domains or unbounded domains satisfying the Poincar\'{e} inequality, there are many results on the existence and long-time behavior of solutions in terms of existence of attractors, see e.g. \cite{AT2013, Gal2015, Damazio2016,  GMR2012, KT2009, Qin2012, YueZhong}. In the whole space, the existence and time decay rates of solutions have been studied in \cite{AT2016, Niche2016, ZZ2015}. However, to the best of our knowledge, there are few works on optimal control problems of Navier-Stokes-Voigt equations, except two recent works \cite{AN2016, AN2017} where the quadratic optimal control and time optimal control problems with distributed controls were investigated. In this paper we continue studying a numerical scheme for the distributed optimal control problem of this system in both time and space variables.

Since the pioneering work \cite{AT1990} in 1990 of Abergel and Temam, optimal control problems for Navier-Stokes equations have been studied by many authors in the last three decades. The analysis of these control problems is well understood, see e.g. \cite{G2003, HK2001, Sritharan1998, TW2006, Wachsmuth} and references therein, where various aspects including first and second order necessary conditions are developed and analyzed. To the contrary, numerical analysis of such optimal control problems is quite limited. This is due to the fact that the restricted regularity of solutions of the evolutionary Navier-Stokes equations, as well as the divergence free condition, and the convective nature of the adjoint equation of the first order necessary condition, pose significant difficulties when analyzing numerical schemes. Standard techniques developed for the numerical analysis of the uncontrolled Navier-Stokes equations cannot be directly applied in the optimal control setting. Furthermore, the presence of control constraints, create many additional difficulties and hence require special techniques involving both first and second order necessary and sufficient conditions. In the literature not many contributions to numerical analysis for control problems with time-dependent Navier-Stokes equations can be found. In \cite{GM2000, GM2000a} Gunzburger and Manservisi presente a numerical approach to control of the instationary Navier-Stokes equations using the first discretize then optimize approach. The first optimize then discretize approach applied to the same problem class is discussed by Hinze in \cite{Hinze2000}. Deckelnick and Hinze provide numerical analysis for a general class of control problems with the instationary Navier-Stokes system in \cite{DH2004}.

In \cite{Casas2012} Casas and Chrysafinos proposed a numerical scheme which is based on the discontinuous time-stepping Galerkin scheme for the piecewise constant time discretization combined with standard conforming finite element subspaces for the discretization in space. They presented space-time error estimates of order $O(h)$, under suitable regularity assumptions on the data, when the controls are discretized by piecewise constants in space and time. Two parameters $\tau$ and $h$ are associated to the numerical scheme (here $\tau$ and $h$, indicating the size of the grids in time and space) and they were needed to satisfy the usual technical assumption $\tau \leq Ch^2$ in order to prove that the discrete equation has a unique solution, and the estimate was optimal in $L^2(0, T; \mathbb{H}^1(\Omega))$ norms for the state and adjoint. The key idea of  \cite{Casas2012} was to utilize ideas from \cite{CMR2007} developed for the stationary Navier-Stokes equations, together with a detailed error analysis of the uncontrolled state and adjoint equations of the underlying scheme. Later in \cite{Casas2015}, Casas and Chrysafinos continued their work in \cite{Casas2012} in the sense that error estimates in $L^2(0, T; \mathbb{L}^2(\Omega))$ of order $O(h^2)$ and $O(h^{3/2-2/p})$ with $p > 3$ depending on the regularity of the target and the initial velocity were proved. We also refer the interested reader to \cite{Casas2016, Casas2017} for some very close related results.

In this paper following the general lines of the approach in \cite{Casas2012, Chrys} we will study a numerical scheme for the optimal control problem $(P)$ of 3D Navier-Stokes-Voigt equations based on the discontinuous time-stepping Galerkin scheme for the piecewise constant time combined with standard conforming finite element subspaces for the discretization in space. It is noticed that in \cite{AN2016} we proved the existence of optimal solutions and established the first and second order optimality conditions for problem $(P)$, where the admissible set is an arbitrary non-empty, convex, closed subset in $\mathbb{L}^2(Q)$. Our main result in the present paper is to derive space-time error estimates under suitable regularity assumptions on the data together with a detailed error analysis of the uncontrolled state and adjoint equations of the underlying scheme. Here the presence of control constraints prevents a direct analysis of the system of state and adjoint state equations. To overcome this difficulty, as in \cite{Casas2012}, we need to use the second order conditions for optimality. The box control constraint assumption ensures that the set of discrete control designed in the paper is convex and closed, which implies that the discrete control problem has a solution. Besides, in our work, since the solution of the Navier-Stokes-Voigt equations is more regular, we are able to prove the uniqueness of the discrete equation without the widely used technical assumption $\tau \leq Ch^2$ as in \cite{Casas2012}, and we obtain that the order of error estimates is $O(\sqrt{\tau}+h)$ instead of $O(h)$ as in \cite{Casas2012}. It is noticed that, under the assumption $\tau\le Ch^2$ as in \cite{Casas2012}, these two orders of errors are the same. To prove the main results, we modify the techniques used in \cite{Casas2012, Chrys} with appropriate adjusts while contributing project operators from the spaces of states and pressure terms to the respective discrete spaces. Because of some technical reasons as in \cite{Casas2012}, the pressure terms must belong to the space $L^2(0,T;H^1(\Omega))$, so we need to assume that the initial state $y_0$ is in $D(A)$ instead of $V$ as in \cite{AN2016}, that is, we consider the strong solutions. It is worthy noticing that, by using a variational discretization in \cite{Hinze2005}, Casas and Chrysafinos can also prove error estimates (in $L^2(0, T; \mathbb{L}^2(\Omega))$)   of order $O(h^2)$ for the distributed optimal control problem of 2D Navier-Stokes equations in \cite{Casas2015}. The proof of such an error estimates for the variational discretization of the optimal control problem $(P)$ will be the goal of a forthcoming paper currently in preparation.

The paper is organized as follows. In Section 2, for convenience of the reader, we recall some auxiliary results on the existence and unique of weak and strong solutions to the system \eqref{NSV1}. We also restate the optimality conditions of the optimal control problem obtained in \cite{AN2016}, however, in this paper these results are transferred into the box control constraint case. The main results of the paper are presented in Section 3, where we analyze the discrete state equations, the discrete adjoint state equations, and we prove the convergence of the discrete control problem and derive the space-time error estimates.

\section{Preliminaries and auxiliary results }
\subsection{Function spaces and inequalities for the nonlinear terms}
Let $\Omega$ be an open bounded domain in $\mathbb{R}^3$ with $C^1$ boundary $\partial\Omega$. For convenience, we set
$$L_0^2(\Omega):=\left\{f\in L^2(\Omega): \int_\Omega f(x)dx=0\right\}, \quad\mathbb{L}^2(\Omega):=L^2(\Omega)^3, $$
$$\mathbb{H}^1(\Omega):=H^1(\Omega)^3,\;\; \mathbb{H}^1_0(\Omega):=H^1_0(\Omega)^3.$$ 
We denote by $\mathbb{H}^{-1}(\Omega)$ the dual spaces of $\mathbb{H}^1_0(\Omega)$. Define
$$
(u,v):=\int\limits_{\Omega} \sum_{j=1}^3 u_jv_j\,  dx, \; \; \; u=(u_1,u_2,u_3),v=(v_1,v_2,v_3)\in \mathbb{L}^2(\Omega),
$$
$$
((u,v)):=\int\limits_{\Omega}\sum_{j=1}^3 \nabla u_j \cdot \nabla v_j\,  dx, \; \;\; u=(u_1,u_2,u_3), v=(v_1,v_2,v_3)\in  \mathbb{H}^1_0(\Omega),$$
and the associated norms $|u|^2:=(u,u), \|u\|^2:=((u,u)).$

Set 
$$ \mathcal{V}=\big\{u\in(C_0^\infty(\Omega))^3\, :\  \nabla\cdot u=0\big\},$$
and denote by $H$ and $V$ the closure of $\mathcal{V}$ in $\mathbb{L}^2(\Omega$ and $\mathbb{H}^1_0(\Omega)$, respectively. Then $H,\,V$ are Hilbert spaces with scalar products $(.,.),\,((.,.))$ respectively.

Let $X$ be a real Banach space with the norm $\|.\|_X$. We denote by $L^p(0,T;X)$ the standard Banach space of all functions from $(0,T)$ to $X$, endowed with the norm
\begin{align*}
\| y \|_{L^p(0,T;X)}&:= \left( \int^{T}_{0}\|y(t)\|_X^p dt \right)^{1/p},\quad 1\le p <\infty,\\
\| y \|_{L^\infty(0,T; X)}&:= \underset{t\in (0,T)}{{\rm ess sup}}\;\|y(t)\|_X.
\end{align*}
When $X$ is a Banach space with the dual space $X'$, we will use $\|.\|_{X'}$ for the norm in $X'$, $\langle .,.\rangle_{X',X}$ for the duality pairing between $X'$ and $X$. In this case, $L^p(0,T;X)$ is also a Banach space, with the dual space being $L^{p'}(0,T;X')$, where $1/p+1/p'=1$. The pairing between $u\in L^{p'}(0,T;X')$ and $v\in L^{p}(0,T;X)$ is
$$\langle u,v\rangle_{L^{p'}(0,T; X'),L^p(0,T; X)}=\int^{T}_{0}\langle u(t),v(t)\rangle_{X',X}dt.$$
To deal with the time derivative in the state equation, we introduce the common space of functions $y$ whose time derivatives $y_t$ exist as abstract functions
$$W^{1,2}(0,T; X):= \{ y\in L^2(0,T;X): y_t \in L^2(0,T;X) \},$$
endowed with the norm
$$ \| y \|_{W^{1,2}(0,T; X)}:=\left( \|y\|^2_{L^2(0,T; X)}+\|y_t\|^2_{L^2(0,T; X)}\right)^{1/2}.$$
When $X$ is a Hilbert space, $L^2(0,T;X)$ and $W^{1,2}(0,T;X)$ are also Hilbert spaces. We will use the following embedding results:
\begin{align*}
&W^{1,2}(0,T; X)\hookrightarrow C([0,T];X) \text{ is continuous (see \cite[p. 190]{Robinson})},\\
&W^{1,2}(0,T;\mathbb{H}^1(\Omega))\hookrightarrow \mathbb{L}^2(Q) \text{ 
is compact (see \cite{Simon})},\\
&W^{1,2}(0,T;\mathbb{H}^1(\Omega))\hookrightarrow C([0,T];\mathbb{L}^2(\Omega)) \text{ 
is compact (see \cite{Simon})}.
\end{align*}
We now define the bilinear and trilinear forms $a:\mathbb{H}^1(\Omega)\times\mathbb{H}^1(\Omega)\to \mathbb{R}$ and $c:\mathbb{H}^1(\Omega)\times\mathbb{H}^1(\Omega)\times\mathbb{H}^1(\Omega)\to \mathbb{R}$ by
\begin{align*}
a(u,v)&=\sum_{i,j=1}^{3}\int_\Omega \partial_{x_i}u_j\partial_{x_i}v_jdx,\\
c(u,v,w)&=\dfrac{1}{2}\left[b(u,v,w)-b(u,w,v)\right]\;\text{with } b(u,v,w)=\sum_{i,j=1}^3\int\limits_\Omega u_i\dfrac{\partial v_j}{\partial x_i}w_j \,dx.
\end{align*}
It is easy to check that if $u\in V, v,w\in \mathbb{H}^1_0(\Omega)$ then $b(u,v,w)=-b(u,w,v)$. Hence
\begin{equation*} 
b(u,v,v)=0,\;\forall\, u\in V, v\in\mathbb{H}^1_0(\Omega).
\end{equation*}
We also define an operator $A:V\to V'$ by $\langle Au,v\rangle=((u,v)),\, u,v\in V.$ Denote by $D(A)$ the domain of $A$, $D(A):=\{u\in H:Au\in H\}$. We have $D(A)=\mathbb{H}^2(\Omega)\cap V.$ The norm in $D(A)$ is defined by $\|u\|_{D(A)}:=|Au|.$
\begin{lemma} \label{trilinear} \cite{ConstantinFoias, Temam} We have
\begin{align}
&b(u,v,w)=c(u,v,w)=-c(u,w,v),\quad\forall\; u\in V,\forall\, v,w\in \mathbb{H}^1_0(\Omega),\notag\\
&|b(u,v,w)|\le C\|u\|\|v\|^{1/2}|Av|^{1/2}|w|,\quad\forall u\in V,v\in D(A),w\in H,\label{b4}\\
&|b(u,v,w)|\le C|u|^{1/4}\|u\|^{3/4}\|v\||w|^{1/4}\|w\|^{3/4},\quad \forall \,u,\,v,\,w\,\in\, \mathbb{H}^1_0(\Omega),\notag\\
&|b(u,v,w)|\le C\|u\|\|v\|\|w\|, \quad \forall \,u,\,v,\,w\,\in\, \mathbb{H}^1_0(\Omega),\notag\\
&|b(u,v,u)|\leq C|u|^{1/2}\|u\|^{3/2}\|v\|, \quad \forall\, u,v\,\in\, \mathbb{H}^1_0(\Omega).\notag\\
&c(u,v,w)=-c(u,w,v), \quad \forall\, u,v,w\in \mathbb{H}^1_0(\Omega),\notag\\
&c(u,v,v)=0, \quad\forall\, u,v\in \mathbb{H}^1_0(\Omega),\notag\\
&|c(u,v,w)|\le C\|u\|\|v\|\|w\|, \quad \forall\, u,v,w\in \mathbb{H}^1_0(\Omega),\notag\\
&|c(u,v,w)|\le C|u|^{1/4}\|u\|^{3/4}\|v\|\|w\|, \quad \forall\, u,v,w\in \mathbb{H}^1_0(\Omega).\notag
\end{align}
\end{lemma} 
\subsection{Existence and uniqueness of solutions to the Navier-Stokes-Voigt equations}
\begin{definition} Let $u\in L^2(0,T;\mathbb{L}^2(\Omega))$ be given. A pair $(y,p)\in W^{1,2}(0,T; V)\times L^2(0,T;L^2_0(\Omega))$ is called a weak solution to the system \eqref{NSV1} on the interval $(0,T)$ if it fulfills
\begin{equation}\label{ST1}
\begin{cases}
(y_t(s),w)+\nu a(y(s),w) +\alpha^2 a(y_t(s),w) +c(y(s),y(s),w)+(p(s),{\rm div}\,w)\\
\hspace{4cm} =(u(s),w), \;\;\forall w\in \mathbb{H}^1_0(\Omega), \text{ for a.e. } s\in [0,T],\\
y(0)=y_0.
\end{cases}
\end{equation}
\end{definition}
\begin{remark} {\rm 
It is clear that if $(y,p)\in W^{1,2}(0,T; V)\times L^2(0,T;L^2_0(\Omega))$ satisfies \eqref{ST1} then $y$ satisfies equations
\begin{equation}
\label{ST2}
\begin{cases}
(y_t(s),w)+\nu a(y(s),w) +\alpha^2 a(y_t(s),w) +c(y(s),y(s),w)=(u(s),w),\\
\hspace{6cm} \forall w\in V,\;\text{ for a.e. } s\in [0,T],\\
y(0)=y_0.
\end{cases}
\end{equation}
Conversely, if $y\in W^{1,2}(0,T;V)$ satisfies \eqref{ST2} then there exists a unique $p\in L^2(0,T;L^2_0(\Omega))$ such that \eqref{ST1} holds.}
\end{remark}
\begin{theorem}\label{tontaiNSV} \cite{AT2013}
For any $y_0\in V$ and $u\in L^2(0,T;\mathbb{L}^2(\Omega))$ given, problem \eqref{ST2} has a unique weak solution $(y,p)$ belonging to $W^{1,2}(0,T;V)\times L^2(0,T;L_0^2(\Omega))$. Furthermore, if $\|u\|_{L^2(0,T;\mathbb{L}^2(\Omega))}\le M$
 then 
\begin{equation}
\label{ST2.1}
\|y\|_{W^{1,2}(0,T;V)}+\|p\|_{L^2(0,T;L^2(\Omega))}\le C_M,
\end{equation}
where $C_M$ is a constant depending on $M$.
\end{theorem}
The existence of strong solutions to problem \eqref{NSV1} is given by the following theorem.
\begin{theorem}
\label{THR2.2}
If $y_0\in D(A)$ and $u\in L^2(0,T;\mathbb{L}^2(\Omega)$ then the unique weak solution $(y,p)$ of \eqref{NSV1} belongs to $W^{1,2}(0,T;D(A))\times L^2(0,T;H^1(\Omega)\cap L^2_0(\Omega))$. Moreover, if $\|u\|_{L^2(0,T;\mathbb{L}^2(\Omega))}\le M$ then
\begin{equation}\label{Pre1.0}
\|y\|_{W^{1,2}(0,T;D(A))}+\|p\|_{L^2(0,T;H^1(\Omega))}\le C_M,\end{equation}
where $C_M$ is a constant depending on $M$.
\end{theorem}
\begin{proof}
The proof is standard by using the Galerkin method, so we only present here some {\it a priori} estimates. Taking $w=y(s)$ in the first equation of \eqref{ST2}, then integrating from $0$ to $t$ we get
\begin{multline}
\label{Pre1.5}
\int^{t}_{0}(y_t(s),y(s))ds + \nu\int^{t}_{0}\|y(s)\|^2ds +\alpha^2\int^{t}_{0}((y_t(s),y(s)))ds\\
 =\int^{t}_{0}(u(s),y(s))ds.
\end{multline}
Here we have used the identity $c(y(s),y(s),y(s))=0$. The right-hand side can be estimated by
\begin{align*}
\left |\int_0^t (u(s),y(s))ds\right | &\le \int_0^t |u(s)|.|y(s)|ds\le C\int_0^t\|y(s)\|.|u(s)|ds\\
&\le \dfrac{\nu}{2}\int_0^t \|y(s)\|^2ds +\dfrac{C^2}{2\nu}\int_0^t |u(s)|^2ds,
\end{align*}
where $C$ depends only on $\Omega$. By integrating by parts, we get from \eqref{Pre1.5} that
\begin{equation*}
\begin{aligned}
|y(t)|^2+\nu \int^{t}_{0}\|y(s)\|^2ds+\alpha^2\|y(t)\|^2 
&\le |y(0)|^2+\alpha^2\|y(0)\|^2+\dfrac{C^2}{\nu}\int^{t}_{0}|u(s)|^2ds\\
&\le |y_0|^2+\alpha^2\|y_0\|^2+\dfrac{C^2}{\nu}\|u\|^2_{\mathbb{L}^2(Q)}.
\end{aligned}
\end{equation*}
So, $y$ belongs to $L^\infty(0,T; V)$ and
\begin{equation}\label{reg1.2}
\|y\|^2_{L^\infty(0,T;V)}\le C(\|y_0\|^2+\|u\|^2_{\mathbb{L}^2(Q)}).
\end{equation}
Now, taking $w=Ay(s)$ in the first equation of \eqref{ST2}, then integrating from $0$ to $t$ we have
\begin{multline}\label{reg1}
\dfrac{1}{2}\|y(t)\|^2-\dfrac{1}{2}\|y_0\|^2 + \nu\int_0^t|A y(s)|^2ds+\dfrac{\alpha^2}{2}|A y(t)|^2-\dfrac{\alpha^2}{2}|A y_0|^2\\
=-\int_0^tb(y(s),y(s),A y(s))ds + \int_0^t(u(s),A y(s))ds.
\end{multline}
We have
\begin{equation}\label{reg1.1}
\left|\int_0^t(u(s),A y(s))ds\right|\le \int_0^t|u(s)||A y(s)|ds \le \dfrac{\nu}{2}\int_0^t|A y(s)|^2ds + C\int_0^t|u(s)|^2ds.
\end{equation}
From inequality \eqref{b4} and the fact that $y\in L^\infty(0,T; V)$ we deduce that
\begin{equation}\label{reg1.0}
\begin{aligned}
\left|\int_0^t b(y(s),y(s),A y(s))ds \right|&\le C\int_0^t\|y(s)\|\|y(s)\|^{1/2}|A y(s)|^{1/2}|A y(s)| ds.\\
&\le C\int_0^t\|y(s)\|^{1/2} |A y(s)|^{3/2}ds\\
&\le \dfrac{\nu}{2}\int_0^t|A y(s)|^2ds + C\int_0^t\|y(s)\|^2ds\\
&\le \dfrac{\nu}{2}\int_0^t|A y(s)|^2ds + C\|y\|^2_{L^2(0,T;V)}.
\end{aligned}
\end{equation}
Using the estimates above, \eqref{reg1} leads to
\begin{equation*}\label{reg2}
|A y(t)|^2\le C\left(|A y_0|^2+\int_0^T |u(s)|^2ds\right)\quad \text{for every } t\in (0,T).
\end{equation*}
Thus we obtain $y\in L^\infty (0,T;D(A))$ and
\begin{equation}\label{reg2.0}
\|y\|^2_{L^\infty(0,T;D(A))}\le C\left(|Ay_0|^2+\int_0^T |u(s)|^2ds\right).
\end{equation}
Now, taking $w=Ay_t(s)$ in the first equation of \eqref{ST2}, then integrating from $0$ to $T$  we have
\begin{multline}
\label{reg3}
\int_0^T\|y_t(s)\|^2ds +\dfrac{\nu}{2}(|Ay(T)|^2-|Ay_0|^2)+\alpha^2\int_0^T|Ay_t(s)|^2ds\\
 =\int_0^T(u(s),Ay_t(s))ds - \int_0^Tb(y(s),y(s),Ay_t(s))ds.
\end{multline}
Using the same arguments as in \eqref{reg1.1} and \eqref{reg1.0} we get 
\begin{align*}
\left|\int_0^T(u(s),A y_t(s))ds\right| &\le \dfrac{\alpha^2}{3}\int_0^T|A y_t(s)|^2ds + C\int_0^T|u(s)|^2ds,\\
\left|\int_0^T b(y(s),y(s),A y_t(s))ds \right| &\le \dfrac{\alpha^2}{3}\int_0^T|A y_t(s)|^2ds + C\|y\|^2_{L^2(0,T;V)}.
\end{align*}
From these estimates and \eqref{reg1.2}, \eqref{reg3} we obtain
$$\int_0^T |Ay_t(s)|^2ds \le C\left(|Ay_0|^2 + \int_0^T|u(s)|^2ds\right).$$
This means that $y_t\in L^2(0,T;D(A))$. Using this estimate and \eqref{reg2.0} we get that
$$\|y\|_{W^{1,2}(0,T;D(A))}\le C_M.$$
The conclusion that $p\in L^2(0,T;H^1(\Omega))$ and that $p$ is bounded in $L^2(0,T;H^1(\Omega))$ by a constant depending on $M$ are direct consequences of the regular properties of the equation $\nabla p =f.$
\end{proof}
Denote by $G$ the control-to-state mapping:
\begin{align*}
G:L^2(0,T;\mathbb{L}^2(\Omega))&\to W^{1,2}(0,T;V)\\
u&\mapsto G(u):=y_u, \text{ the unique solution of \eqref{ST2}.}
\end{align*}
From now on, for convenience, we sometimes write $J$ as a functional of control variable $u$ as follows $J(u):=J(G(u), u)$.
\begin{theorem}
\label{THR2.3}
If $y_0\in D(A)$ and $u_n\rightharpoonup u$ in the space $L^2(0,T;\mathbb{L}^2(\Omega))$, then from the sequence $\{y_{u_n}\}_n$ we can extract a subsequence, denoted in the same way, such that $y_{u_n}\rightharpoonup y_u$ in the space $W^{1,2}(0,T;D(A))$. 
\end{theorem}
\begin{proof}
It follows from Theorem \ref{THR2.2} that the sequence $\{y_{u_n}\}_n$ is bounded in the space $W^{1,2}(0,T;D(A))$. Hence we can extract a subsequence converging weakly to some $y\in W^{1,2}(0,T;D(A))$. Then we can easily pass to the limit in \eqref{ST2} to show that $y\equiv y_u$.
\end{proof}
\subsection{Optimality conditions}
Now, we restate the results obtained in \cite{AN2016} when the set of admissible controls is the box $U_{\alpha,\beta}$. % given by \eqref{AdSet}. % instead of being an arbitrary non-empty closed convex subset of $L^2(0,T;\mathbb{L}^2(\Omega))$.
\begin{theorem}
The mapping $G$ is of class $C^2$. If we set $z_v:=G'(u)v$, $z_{vv}=G''(u)v^2$ then $z_v,z_{vv}$ are respectively the unique solutions of the following equations:
\begin{equation*}
\begin{cases}
(z_{vt},w)+\nu a(z_v,w)+\alpha^2a(z_{vt},w)+c(z_v,y_u,w)+c(y_u,z_v,w)=(v,w),\quad\forall w\in V,\\
z_v(0)=0,
\end{cases}
\end{equation*}
\begin{equation*}
\begin{cases}
(z_{vvt},w)+\nu a(z_{vv},w)+\alpha^2a(z_{vvt},w)+c(z_{vv},y_u,w)+c(y_u,z_{vv},w)\\
\hspace{6cm} =-2c(z_v,z_v,w),\quad\forall w\in V,\\
z_{vv}(0)=0.
\end{cases}
\end{equation*}
Furthermore, if $\|y_u\|_{W^{1,2}(0,T;V)}\le M$ then 
\begin{equation}
\label{DGE2.1}
\|z_v\|_{W^{1,2}(0,T;V)}\le C_M\|v\|_{L^2(0,T;\mathbb{L}^2(\Omega))},
\end{equation}
where $C_M$ is a constant depending on $M$.
\end{theorem}
\begin{theorem}
The cost functional $J:L^2(0,T;\mathbb{L}^2(\Omega))\to \mathbb{R}$ is of class $C^2$. The first and second order derivatives of $J$ are given by
\begin{align*}
J'(u)v= &\int_0^T\int_\Omega (\lambda+\gamma u)vdxdt,\\
J''(u)v^2=&\alpha_T\int_\Omega|z_v(T)|^2dx+\alpha_Q\int_0^T\int_\Omega|z_v|^2dxdt+\gamma\int_0^T\int_\Omega|v|^2dxdt\\
&-2\int_0^Tc(z,z,\lambda)dt,
\end{align*}
where $\lambda$ is a weak solution of the following system
\begin{equation}
\label{DGE3}
\begin{cases}
-\lambda_t- \nu\Delta\lambda +\alpha^2\Delta\lambda _t-(y_u\cdot\nabla)\lambda + (\nabla y_u)^T\lambda+\nabla q= \alpha _Q(y_u-y_Q), \;x\in \Omega, t>0,\\
\hfil \nabla \cdot  \lambda=0,\; x\in \Omega, t>0,\\
\hfil \lambda(x,t)=0, \; x\in\partial\Omega, t>0,\\
\hfil \lambda(T)-\alpha^2\Delta \lambda(T)+\nabla r=\alpha _T(y_u(T)-y_T),\; x\in\Omega.
\end{cases}
\end{equation}
\end{theorem}
\begin{remark}{\rm 
\begin{itemize}
\item[i.] A triplet $(\lambda,q,r)\in W^{1,2}(0,T;V)\times L^2(0,T;L^2_0(\Omega))\times L^2_0(\Omega)$ is called a weak solution to the system \eqref{DGE3} on the interval $(0,T)$ if for every $w\in \mathbb{H}^1_0(\Omega)$, it fulfills
\begin{equation}
\label{DGE4}
\begin{cases}
-(\lambda_t,w)+\nu a(\lambda,w)-\alpha^2a(\lambda_t,w)+c(y_u,w,\lambda)+c(w,y_u,\lambda)+(q,{\rm div}\, w)\\
\hspace{9cm}=\alpha_Q(y_u-y_Q,w),\\
(\lambda(T),w)+\alpha^2a(\lambda(T),w)+(r,{\rm div}\,w)=\alpha_T(y_u(T)-y_T,w).
\end{cases}
\end{equation}
\item[ii.] If $(\lambda,q,r)\in W^{1,2}(0,T;V)\times L^2(0,T;L^2_0(\Omega))\times L^2_0(\Omega)$ is a weak solution of the system \eqref{DGE3} then for every $w\in V$, $\lambda$ satisfies the following equations
 \begin{equation}
\label{DGE5}
\begin{cases}
-(\lambda_t,w)+\nu a(\lambda,w)-\alpha^2a(\lambda_t,w)+c(y_u,w,\lambda)+c(w,y_u,\lambda)=\alpha_Q(y_u-y_Q,w),\\
(\lambda(T),w)+\alpha^2a(\lambda(T),w)=\alpha_T(y_u(T)-y_T,w).
\end{cases}
\end{equation}
Conversely, if $\lambda\in W^{1,2}(0,T;V)$ satisfies \eqref{DGE5} then there exists a unique pair $(q,r)\in L^2(0,T;L^2_0(\Omega))\times L^2_0(\Omega)$ such that \eqref{DGE4} holds for every $w\in \mathbb{H}^1_0(\Omega)$.
\item[iii.] The system \eqref{DGE3} has a unique weak solution $(\lambda,q,r)\in W^{1,2}(0,T;V)\times L^2(0,T;L^2_0(\Omega))\times L^2_0(\Omega)$. One can prove that this solution belongs to the space $W^{1,2}(0,T;D(A))\times L^2(0,T;H^1(\Omega)\cap L^2_0(\Omega))\times (H^1(\Omega)\cap L^2_0(\Omega)).$ Furthermore, if $\|u\|_{L^2(0,T;\mathbb{L}^2(\Omega))}\le M$ then 
$$\|\lambda\|_{W^{1,2}(0,T;D(A))}+\|q\|_{L^2(0,T;H^1(\Omega))}+\|r\|_{H^1(\Omega)}\le C_M,$$
where $C_M$ is a constant depending on $M$.
\end{itemize}}
\end{remark}
\begin{theorem}
Let $\bar{u}\in L^2(0,T;\mathbb{L}^2(\Omega))$ be a local optimal control with associated state $\bar{y}\in W^{1,2}(0,T;D(A))$. Then
\begin{equation*}
\int_0^T\int_\Omega (\bar{\lambda}+\gamma\bar{u})\cdot (v-\bar{u})dxdt\ge 0,\quad \forall v\in U_{\alpha, \beta},
\end{equation*}
where $\bar{\lambda}$ is the adjoint state, i.e. $(\bar{\lambda},\bar{q},\bar{r})$ is the unique weak solution of the following system on the interval $(0,T)$
\begin{equation*}
\begin{cases}
-\bar{\lambda}_t- \nu\Delta\bar{\lambda} +\alpha^2\Delta\bar{\lambda} _t-(\bar{y}\cdot\nabla)\bar{\lambda} + (\nabla \bar{y})^T\bar{\lambda}+\nabla \bar{q}= \alpha _Q(\bar{y}-y_Q), \;x\in \Omega, t>0,\\
\hfil \nabla \cdot  \bar{\lambda}=0,\; x\in \Omega, t>0,\\
\hfil \bar{\lambda}(x,t)=0, \; x\in\partial\Omega, t>0,\\
\hfil \bar{\lambda}(T)-\alpha^2\Delta \bar{\lambda}(T)+\nabla \bar{r}=\alpha _T(\bar{y}(T)-y_T),\; x\in\Omega.
\end{cases}
\end{equation*}
\end{theorem}
To state the second-order optimality conditions we define the cone of critical directions as follows
$$\mathcal{C}_{\bar{u}}=\left\{v\in L^2(0,T;\mathbb{L}^2(\Omega)): v \text{ satisfies } \eqref{DGE8}, \eqref{DGE9}, \eqref{DGE10} \right\}.$$
\begin{align}
&v_j(t,x)\ge 0 \text{ if } -\infty < \alpha_j=\bar{u}_j(t,x),\;j=1,2,3,\label{DGE8}\\
&v_j(t,x)\le 0 \text{ if } \bar{u}_j(t,x)=\beta_j<+\infty,\;j=1,2,3, \label{DGE9}\\
&v_j(t,x)=0 \text{ if } \bar{\lambda}+\gamma\bar{u}\ne 0,\;j=1,2,3.\label{DGE10}
\end{align}
The following theorem gives the second-order necessary and sufficient optimality conditions.
\begin{theorem}
If $\bar{u}$ be a local solution of problem $(P)$ then
$$J''(\bar{u})v^2\ge 0\;\;\forall v\in \mathcal{C}_{\bar{u}}.$$
Conversely, if $\bar{u}\in U_{\alpha, \beta}$ satisfies
\begin{align}
&J'(\bar{u})(u-\bar{u})\ge 0\;\; \forall u\in U_{\alpha, \beta},\notag\\
&J''(\bar{u})v^2>0\;\;\forall v\in \mathcal{C}_{\bar{u}}\backslash \{0\},\label{DGE12}
\end{align}
then there exist $\varepsilon>0$ and $\rho>0$ such that 
\begin{equation*}
J(u)\ge J(\bar{u})+\varepsilon\|u-\bar{u}\|^2_{L^2(0,T;\mathbb{L}^2(\Omega))}\;\;\forall u\in U_{\alpha, \beta}\cap \bar{B}_\rho(\bar{u}).
\end{equation*}
\end{theorem}

\section{Numerical approximation of the optimal control problem}
Let $\{\mathcal{T}_h\}_{h>0}$ be a family of triangulation of $\overline{\Omega}$, defined in the standard way. To each element $T\in\mathcal{T}_h$, we denote by $h_T$ and $\rho_T$ the diameter of the set $T$ and diameter of the largest ball contained in $T$. Define the size of the mesh by $h$, i.e. $h=\max_{T\in \mathcal{T}_h}h_T.$ We also assume that the following standard regularity assumptions on the triangulation hold:
\begin{itemize}
\item[(i)] There exist two positive constants $\rho_\mathcal{T},\delta_\mathcal{T}$ such that $\dfrac{h_T}{\rho_T}\le \rho_\mathcal{T}$ and $\dfrac{h}{h_T}\le \delta_\mathcal{T}$ for every $T\in\mathcal{T}_h$ and for every $h>0.$
\item[(ii)] Set $\overline{\Omega}_h=\cup_{T\in\mathcal{T}_h}T$ and denote by $\Omega_h$ and $\Gamma_h$ its interior and its boundary, respectively. We assume that the vertices of $\mathcal{T}_h$ placed on the boundary $\Gamma_h$ are points of $\Gamma$.
\end{itemize} 
Since $\Omega$ is convex, from the last assumption we have that $\Omega_h$ is also convex. Moreover, we assume that
\begin{equation}
\label{Ap1}
|\Omega\backslash \Omega_h|\le Ch^2.
\end{equation}
On the mesh $\mathcal{T}_h$ we consider two finite dimensional spaces $Z_h\subset \mathbb{H}^1_0(\Omega)$ and $Q_h\subset L^2_0(\Omega)$ formed by piecewise polynomials in $\Omega_h$ and vanishing in $\Omega\backslash \Omega_h.$ We make the following assumptions on these spaces:
\begin{itemize}
\item[(A1)] If $z\in \mathbb{H}^{1+l}(\Omega)\cap \mathbb{H}_0^1(\Omega)$ then 
\begin{equation*}
\inf_{z\in Z_h}\|z-z_h\|_{{\mathbb{H}}^s(\Omega)}\le Ch^{l+1-s}\|z\|_{\mathbb{H}^{1+l}(\Omega)} \text{ for } 0\le l\le 1 \text{ and } s=0,1.
\end{equation*}
\item[(A2)] If $q\in H^1(\Omega)\cap L^2_0(\Omega)$ then 
\begin{equation*}
\inf_{q_h\in Q_h}|q-q_h|\le Ch\|q\|_{H^1(\Omega)}.
\end{equation*}
\item[(A3)] The subspaces $Z_h$ and $Q_h$ satisfy the inf-sup condition: $\exists \beta>0$ such that 
\begin{equation*}
\inf_{q_h\in Q_h}\sup_{z_h\in Z_h}\dfrac{\mathfrak{b}(z_h,q_h)}{\|z_h\|_{{\mathbb{ H}}^1(\Omega)}\|q_h\|_{L^2(\Omega)}}\ge \beta,
\end{equation*}
where $\mathfrak{b}:\mathbb{H}^1(\Omega)\times L^2(\Omega)\to \mathbb{R}$ is defined by 
$$\mathfrak{b}(z,q)=\int_{\Omega} q(x)\text{div}z(x)dx.$$
\end{itemize}
These assumptions are satisfied by the usual finite elements considered in the discretization of Navier-Stokes equations, see \cite[Chapter 2]{Girault}.

We also consider a subspace $V_h$ of $Z_h$ defined by 
$$V_h=\{y_h\in Z_h:\mathfrak{b}(y_h,q_h)=0\;\;\forall q_h\in Q_h\},$$
and set
$$U_h=\{u_h\in \mathbb{L}^2(\Omega_h):u_h|_T\equiv u_T\in \mathbb{R}^3\;\;\forall T\in \mathcal{T}_h\}.$$
Now, we consider the discretization in time. Let $0=t_0<t_1<\cdots <t_{N_\tau}=T$ be a partition of interval $[0,T]$. We denote $\tau_n=t_n-t_{n-1}.$ We make the following assumption:
\begin{equation*}
\exists\, \rho_0>0 \text{ such that } \tau = \max_{1\le n\le N_\tau} \tau_n<\rho_0\tau_n\;\;\forall\, 1\le n\le N_\tau \text{ and } \forall \tau > 0.
\end{equation*}
Given a triangulation $\mathcal{T}_h$ of $\Omega$ and a grid of points $\{t_n\}_{n=1}^{N_\tau}$ of $[0,T]$, we set $\sigma=(\tau,h)$. We consider the subspace of functions that are piecewise constant in time 
$$U_\sigma=\{u_\sigma\in L^2(0,T;U_h): u_\sigma|_{(t_{n-1},t_n)}\in U_h\, \text{ for } 1\le n\le N_{\tau}\}.$$
We seek for the discrete controls in the space $U_\sigma$. An element of this space can be written in the form
\begin{equation*}
u_\sigma = \Sigma_{n=1}^{N_\tau}\Sigma_{T\in \mathcal{T}_h}u_{n,T}\chi_n\chi_T, \,\, \text{ with } u_{n,T}\in \mathbb{R}^3,
\end{equation*}
where $\chi_n$ and $\chi_T$ are the characteristic functions of $(t_{n-1},t_n)$ and $T$, respectively. Therefore, the dimension of $U_\sigma$ is $3N_\tau N_h$, where $N_h$ is the number of elements in $\mathcal{T}_h$. In $U_\sigma$ we consider the convex subset 
$$U_{\sigma,ad}=U_\sigma\cap U_{\alpha,\beta}=\{u_\sigma\in U_\sigma:u_{n,T}\in \Pi_{i=1}^3[\alpha_i,\beta_i]\}.$$
For each given sequence $(y_{0,h},y_{1,h},\ldots,y_{N_\tau,h})\in V_h^{N_\tau+1}$, we define a function $y_\sigma:[0,T]\to V_h$ by
\begin{equation*}
\begin{cases}
y_\sigma(t_n)=y_{n,h},\;n=0,1,\ldots,N_\tau,\\
y_\sigma(t)=y_{n,h}\;\forall t\in (t_{n-1},t_n).
\end{cases}
\end{equation*}
We denote by $V_\sigma$ the set of all functions that are defined by this way.

Now, we consider the numerical discretization of the state equations \eqref{ST1}. We will use a discontinuous time-stepping Galerkin method, with piecewise constants in time and conforming finite element spaces in space. For $u\in \mathbb{L}^2(Q)$, the discrete state equation is given by
\begin{equation}
\label{DSE1}
\begin{cases}
\text{ For } n=1,2,\ldots,N_\tau,\\
\left(\dfrac{y_{n,h}-y_{n-1,h}}{\tau_n},w_h\right)+\nu a(y_{n,h},w_h)+\alpha^2 a\left(\dfrac{y_{n,h}-y_{n-1,h}}{\tau_n},w_h\right)\\
\hspace{5cm}+c(y_{n,h},y_{n,h},w_h)=(u_n,w_h),\;\;\forall w_h\in V_h,\\
y_{0,h}=P_hy_0,
\end{cases}
\end{equation}
where $(u_n,w_h)=\dfrac{1}{\tau_n}\int\limits_{t_{n-1}}^{t_n}(u(t),w_h)dt$ and $P_hy_0$ is defined in Definition \ref{Def1}.

We will prove later that for every $u\in\mathbb{L}^2(Q)$, problem \eqref{DSE1} has a unique solution $y_\sigma(u)\in V_\sigma$. Now we can define the discrete control problem as follows
\begin{equation*}
(P_\sigma)\hspace{0.5cm}
\begin{cases}
\min J_\sigma (u_\sigma)\\
u_\sigma\in U_{\sigma,ad},
\end{cases}
\end{equation*}
with
\begin{multline*}
J_\sigma(u_\sigma)=\dfrac{\alpha_T}{2}\int_{\Omega_h}|y_\sigma(u_\sigma)(T)-y_T^h|^2dx+\dfrac{\alpha_Q}{2}\int_0^T\int_{\Omega_h}|y_\sigma(u_\sigma)-y_Q|^2dxdt\\
+\dfrac{\gamma}{2}\int_0^T\int_{\Omega_h}|u_\sigma|^2dxdt,
\end{multline*}
where $y_T^h\in V_h$ satisfies the following condition
\begin{equation*}
\exists C>0 \text{ such that } \|y_T-y_T^h\|_{\mathbb{L}^2(\Omega_h)}\le Ch \text{ and } \|y_T^h\|_{\mathbb{H}^1(\Omega_h)}\le C, \;\forall h>0.
\end{equation*}
The outline of this section is as follows. In Subsection 3.1, we analyze the discrete state equations \eqref{DSE1}. Then we study the discrete adjoint state equations in Subsection 3.2. Finally, we prove the convergence of solutions to problem $(P_\sigma)$ and derive the error estimates for the discretization in the last subsection.
\subsection{Analysis of the discrete state equations}
By a standard argument, using the identify $c(u,v,v)=0 \;\forall u,v\in \mathbb{H}^1_0(\Omega)$ and Brouwer's fixed-point theorem, one can easily prove that system \eqref{DSE1} has at least one solution. In this subsection, we will prove that the solution is unique. According to an abstract approximation result (see \cite{Girault}), for given $y\in V$, $p\in L^2_0(\Omega)$, the following problems have unique solutions.
\begin{align*}
(Pr_{1})&\quad\text{Find } y_h\in V_h \text{ satisfying: }
a_\alpha(y_h,v_h)=a_\alpha(y,v_h)\;\;\forall v_h\in V_h,\\
(Pr_{2})&\quad \text{Find a pair } (y_h,p_h)\in V_h\times Q_h \text{ satisfying:}\\
 &\quad a_\alpha(y_h,v_h)+\mathfrak{b}(v_h,p_h)=\mathfrak{b}(v_h,p)\;\; \forall v_h\in Z_h.
\end{align*}
Here, $a_\alpha(u,v)=(u,v)+\alpha^2(\nabla u,\nabla v).$ These results allow us to give the following definition. 
\begin{definition}
\label{Def1}
We define operators

\begin{tabular}{rcrl}
$P_h:$& $V$&$\to$ &$V_h$\\
&$y$&$\mapsto$& $y_h,\text{ which is the unique solution of the problem } (Pr_1),$
\end{tabular}

\begin{tabular}{rcrl}
$R_h:$&$L^2_0(\Omega)$&$\to$&$ Q_h$\\
&$p$&$\mapsto$&$ p_h, \text{ which is the second component of the solution of } (Pr_2)$.
\end{tabular}

We also define $P_\sigma:C([0,T];V)\to V_\sigma$ by $(P_\sigma y)_{n,h}=P_hy(t_n)$ for $0\le n\le N_\tau.$
\end{definition}
Obviously, there exists a constant $C$ depending only on $\alpha$ such that $\|P_hy\|_{\mathbb{H}^1(\Omega)}\le C\|y\|_{\mathbb{H}^1(\Omega)},\, \forall y\in V$. Using Theorem 1.1 in \cite[Chapter 2]{Girault}, we can easily get the following lemma from assumptions $(A1)-(A3)$.
\begin{lemma}
\label{LM1}
For every $u\in V\cap \mathbb{H}^2(\Omega)$, $p\in L^2_0(\Omega)\cap H^1(\Omega)$, we have
\begin{align*}
\|u-P_hu\|_{\mathbb{H}^1(\Omega)}&\le Ch\|u\|_{\mathbb{H}^2(\Omega)},\\
|p-R_hp|&\le Ch\|p\|_{H^1(\Omega)},
\end{align*}
where $C$ is a constant independent of $h$.
\end{lemma}
\begin{lemma}
\label{LM2}
There exists a constant $C>0$ independent of $\sigma$ such that for every $y\in W^{1,2}(0,T;D(A))$ we have
$$\|y-P_\sigma y\|_{L^2(0,T;\mathbb{H}^1(\Omega))}\le C\left(h\|y\|_{L^2(0,T;\mathbb{H}^2(\Omega))}+\tau\|y'\|_{L^2(0,T;\mathbb{H}^1(\Omega))}\right).$$
\end{lemma}
\begin{proof} We have
\begin{align*}
&\|y-P_\sigma y\|_{L^2(0,T;\mathbb{H}^1(\Omega))}
=\left\{\sum_{n=1}^{N_\tau}\int_{t_{n-1}}^{t_n}\|y(t)-P_h y(t_n)\|_{\mathbb{H}^1(\Omega)}^2dt\right\}^{1/2}\\
&\le \left\{\sum_{n=1}^{N_\tau}\int_{t_{n-1}}^{t_n}\|y(t)-P_h y(t)\|_{\mathbb{H}^1(\Omega)}^2dt\right\}^{1/2}\\
&\hspace{5cm} + \left\{\sum_{n=1}^{N_\tau}\int_{t_{n-1}}^{t_n}\|P_hy(t)-P_h y(t_n)\|_{\mathbb{H}^1(\Omega)}^2dt\right\}^{1/2}\\
&\le Ch\left\{\sum_{n=1}^{N_\tau}\int_{t_{n-1}}^{t_n}\|y(t)\|_{\mathbb{H}^2(\Omega)}^2dt\right\}^{1/2}+C\left\{\sum_{n=1}^{N_\tau}\int_{t_{n-1}}^{t_n}\|y(t)-y(t_n)\|_{\mathbb{H}^1(\Omega)}^2dt\right\}^{1/2}\\
&= Ch\left\{\sum_{n=1}^{N_\tau}\int_{t_{n-1}}^{t_n}\|y(t)\|_{\mathbb{H}^2(\Omega)}^2dt\right\}^{1/2}+C\left\{\sum_{n=1}^{N_\tau}\int_{t_{n-1}}^{t_n}\|\int_{t}^{t_n}y'(s)ds\|_{\mathbb{H}^1(\Omega)}^2dt\right\}^{1/2}\\
&\le Ch\left\{\sum_{n=1}^{N_\tau}\int_{t_{n-1}}^{t_n}\|y(t)\|_{\mathbb{H}^2(\Omega)}^2dt\right\}^{1/2}\\
&\hspace{4cm}+C\left\{\sum_{n=1}^{N_\tau}\int_{t_{n-1}}^{t_n}(t_n-t)\int_{t_{n-1}}^{t_n}\|y'(s)\|^2_{\mathbb{H}^1(\Omega)}dsdt\right\}^{1/2}\\
&\le C(h\|y\|_{L^2(0,T;\mathbb{H}^2(\Omega))}+\tau \|y'\|_{L^2(0,T;\mathbb{H}^1(\Omega))}).
\end{align*}
\end{proof}
\begin{lemma}
\label{LM53}
Let $y\in W^{1,2}(0,T;D(A))$ be the unique solution of \eqref{ST2}. We consider the following system
\begin{equation}
\label{DS2}
\begin{cases}
\text{For } n=1,\ldots,N_\tau,\\
 \left (\dfrac{\hat{y}_{n,h}-\hat{y}_{n-1,h}}{\tau_n},w_h\right )+\nu a(\hat{y}_{n,h},w_h)+\alpha^2a(\dfrac{\hat{y}_{n,h}-\hat{y}_{n-1,h}}{\tau_n},w_h)\\
 \hspace{8.5cm}=(f_n,w_h)\;\;\forall\, w_h\in V_h,\\
 \hat{y}_{0,h}=y_{0h},  
\end{cases}
\end{equation}
where 
$$(f_n,w_h)=\dfrac{1}{\tau_n}\int_{t_{n-1}}^{t_n}\{\nu a(y(t),w_h)+\alpha^2a(y'(t),w_h)+(y'(t),w_h)\}dt.$$
This system has a unique solution $\hat{y}_\sigma\in V_\sigma$. Moreover, we have the following properties:
\begin{enumerate}
\item $\{\hat{y}_\sigma\}_\sigma$ is bounded in $L^\infty(0,T;\mathbb{H}^1(\Omega)).$
\item There exists a constant $C>0$ independent of $\sigma$ such that 
\begin{multline}
\label{DS3}
\max_{1\le n\le N_\tau}\|y(t_n)-\hat{y}_\sigma(t_n)\|_{\mathbb{H}^1(\Omega)} + 
\|y-\hat{y}_\sigma\|_{L^2(0,T;\mathbb{H}^1(\Omega))}\\
\le C(h\|y\|_{C([0,T];\mathbb{H}^2(\Omega))}+\tau\|y'\|_{L^2(0,T;\mathbb{H}^1(\Omega))}),
\end{multline}
\begin{multline}
\label{DS3.1}
\|y-\hat{y}_\sigma\|_{L^\infty(0,T;\mathbb{H}^1(\Omega))}
\le C\{(\tau+\sqrt{\tau})\|y'\|_{L^2(0,T;\mathbb{H}^1(\Omega))}+h\|y\|_{C([0,T];\mathbb{H}^2(\Omega))}\}.
\end{multline}
\end{enumerate}
\end{lemma}
\begin{proof}
The existence and uniqueness of the solution $\hat{y}_\sigma$ is easily proved by using the Lax-Milgram theorem. We are going to prove the boundedness of $\hat{y}_\sigma$. Taking $w_h=\hat{y}_{n,h}$ in \eqref{DS2} we have
\begin{multline}
\label{DS4}
|\hat{y}_{n,h}|^2+\nu\tau_n|\nabla\hat{y}_{n,h}|^2+\alpha^2|\nabla\hat{y}_{n,h}|^2\\
=(\hat{y}_{n-1,h},\hat{y}_{n,h})+\alpha^2(\nabla\hat{y}_{n-1,h},\nabla \hat{y}_{n,h})+\tau_n(f_n,\hat{y}_{n,h}).
\end{multline}
It follows from H\"older's inequality that
\begin{equation}
\label{DS5}
(\hat{y}_{n-1,h},\hat{y}_{n,h})\le \dfrac{1}{2} |\hat{y}_{n-1,h} |^2 +\dfrac{1}{2} |\hat{y}_{n,h} |^2 ,
\end{equation}
\begin{equation}
\label{DS6}
\alpha^2(\nabla\hat{y}_{n-1,h},\nabla\hat{y}_{n,h})\le \dfrac{\alpha^2}{2} |\nabla\hat{y}_{n-1,h} |^2 +\dfrac{\alpha^2}{2} |\nabla\hat{y}_{n,h} |^2,
\end{equation}
$$
\int_{t_{n-1}}^{t_n}\nu a(y(t),\hat{y}_{n,h})dt\le \dfrac{\nu}{4}\tau_n  |\nabla\hat{y}_{n,h} |^2 +C_{\nu} \int_{t_{n-1}}^{t_n} |\nabla y(t) |^2dt,
$$
$$
\int_{t_{n-1}}^{t_n}\alpha^2a(y'(t),\hat{y}_{n,h})dt\le
 \dfrac{\nu}{4}\tau_n  |\nabla\hat{y}_{n,h} |^2 +C_{\nu,\alpha} \int_{t_{n-1}}^{t_n} |\nabla y'(t) |^2 dt,
 $$
 $$
 \int_{t_{n-1}}^{t_n}(y'(t),\hat{y}_{n,h})dt\le\dfrac{\nu}{4}\tau_n  |\nabla\hat{y}_{n,h} |^2 +C_{\nu,\Omega} \int_{t_{n-1}}^{t_n} |y'(t)|^2dt.
$$
In the last estimate, we have used the fact that 
$|\hat{y}_{n,h}| \le C_\Omega |\nabla\hat{y}_{n-1,h}|,$ 
by Poincar\'e's inequality, since $\hat{y}_{n,h}\in \mathbb{H}^1_0(\Omega)$. Summarizing the last three estimates leads to
\begin{equation}
\label{DS7}
\tau_n(f_n,\hat{y}_{n,h})\le \dfrac{3\nu}{4}\tau_n|\nabla\hat{y}_{n,h}|^2 +C_{\nu,\alpha,\Omega}\|y\|^2_{W^{1,2}(t_{n-1},t_n;\mathbb{H}^1(\Omega))}.
\end{equation}
Combining \eqref{DS4} with \eqref{DS5}, \eqref{DS6}, \eqref{DS7} gives
\begin{align*}
&\dfrac{1}{2} |\hat{y}_{n,h}|^2  +\dfrac{\nu}{4}\tau_n |\nabla\hat{y}_{n,h} |^2 +\dfrac{\alpha^2}{2} |\nabla\hat{y}_{n,h} |^2 \\
\le & \dfrac{1}{2} |\hat{y}_{n-1,h} |^2+\dfrac{\alpha^2}{2} |\nabla\hat{y}_{n-1,h}|^2 +C_{\nu,\alpha,\Omega}\|y\|^2_{W^{1,2}(t_{n-1},t_n;\mathbb{H}^1(\Omega))}.
\end{align*}
Summarizing these estimates from $n=1$ to $n=k$ ($k$ is an arbitrary integer in the set $\{1,2,\ldots,N_\tau\}$) we obtain
\begin{equation}
\label{DS7.1}
\|\hat{y}_\sigma\|_{L^\infty(0,T;\mathbb{H}^1(\Omega))}\le C(\|y_{0h}\|_{\mathbb{H}^1(\Omega)}+\|y\|_{W^{1,2}(0,T;\mathbb{H}^1(\Omega))}),
\end{equation}
which gives the first statement. To prove the second statement, we set
$$e=y-\hat{y}_\sigma,\quad e_h=P_\sigma y -\hat{y}_\sigma, \quad e_p= y-P_\sigma y.$$
Since $\hat{y}_{n,h}=y(t_n)-e(t_n)$, \eqref{DS2} gives
\begin{multline*}
\left(\dfrac{e(t_n)-e(t_{n-1})}{\tau_n},w_h\right)+\dfrac{1}{\tau_n}\int_{t_{n-1}}^{t_n}\nu a(e(t),w_h) + \alpha^2a\left(\dfrac{e(t_n)-e(t_{n-1})}{\tau_n},w_h\right)=0.
\end{multline*} 
Replacing $e$ by $e_p+e_h$ and using the definition of $P_h$ yield
\begin{multline*}
(e_h(t_n)-e_h(t_{n-1}),w_h)+\alpha^2a(e_h(t_n)-e_h(t_{n-1}),w_h)+\int_{t_{n-1}}^{t_n}\nu a(e_h(t),w_h)dt\\
+\int_{t_{n-1}}^{t_n}\nu a(e_p(t),w_h)dt=0.
\end{multline*}
Taking $w_h=e_h(t_n)$ we have
\begin{multline*}
\dfrac{1}{2}|e_h(t_n)|^2-\dfrac{1}{2}|e_h(t_{n-1})|^2+\dfrac{1}{2}|e_h(t_n)-e_h(t_{n-1})|^2
 + \dfrac{\alpha^2}{2}|\nabla e_h(t_n)|^2\\
 -\dfrac{\alpha^2}{2}|\nabla e_h(t_{n-1})|^2+\dfrac{\alpha^2}{2}|\nabla e_h(t_n)-\nabla e_h(t_{n-1})|^2
 +\nu \int_{t_{n-1}}^{t_n}|\nabla e_h|^2dt\\
 \le \dfrac{\nu}{2}\|y-P_\sigma y\|^2_{L^2(t_{n-1},t_n;\mathbb{H}^1(\Omega))}
 +\dfrac{\nu}{2} \int_{t_{n-1}}^{t_n}|\nabla e_h|^2dt.
\end{multline*}
Therefore,
\begin{multline*}
\dfrac{1}{2}|e_h(t_n)|^2+\dfrac{1}{2}|e_h(t_n)-e_h(t_{n-1})|^2
 + \dfrac{\alpha^2}{2}|\nabla e_h(t_n)|^2
 +\dfrac{\alpha^2}{2}|\nabla e_h(t_n)-\nabla e_h(t_{n-1})|^2\\
 +\dfrac{\nu}{2} \int_{t_{n-1}}^{t_n}|\nabla e_h|^2dt
 \le\dfrac{1}{2}|e_h(t_{n-1})|^2+\dfrac{\alpha^2}{2}|\nabla e_h(t_{n-1})|^2
 + \dfrac{\nu}{2}\|y-P_\sigma y\|^2_{L^2(t_{n-1},t_n;\mathbb{H}^1(\Omega))}.
\end{multline*}
Adding these inequalities for $n=1,\ldots,k$, and noticing that $e_h(0)=0$, we have
\begin{multline*}
\dfrac{1}{2}|e_h(t_k)|^2+\dfrac{\alpha^2}{2}|\nabla e_h(t_k)|^2+\dfrac{\nu}{2}\int_{0}^{t_k}|\nabla e_h|^2dt
\le \dfrac{\nu}{2}\|y-P_\sigma y\|^2_{L^2(0,T;\mathbb{H}^1(\Omega))}.
\end{multline*}
This implies that, for every $k\in\{1,2, \ldots,N_\tau\}$
\begin{align}
\|e_h(t_k)\|_{\mathbb{H}^1(\Omega)}&\le C\{h\|y\|_{L^2(0,T;\mathbb{H}^2(\Omega))}+\tau\|y'\|_{L^2(0,T;\mathbb{H}^1(\Omega))}\},\label{DS8}\\
\|e_h\|_{L^2(0,T;\mathbb{H}^1(\Omega))}&\le C\{h\|y\|_{L^2(0,T;\mathbb{H}^2(\Omega))}+\tau\|y'\|_{L^2(0,T;\mathbb{H}^1(\Omega))}\},\label{DS9}
\end{align}
by Lemma \ref{LM2}. Now, we are going to estimate $\|e_p(t_k)\|_{\mathbb{H}^1(\Omega)}$. We have
$e_p(t_k)=y(t_k)-P_hy(t_k)$. Then,
\begin{equation}
\label{DS10}
\|e_p(t_k)\|_{\mathbb{H}^1(\Omega)}\le Ch\|y(t_k)\|_{\mathbb{H}^2(\Omega)}\le Ch\|y\|_{C([0,T];\mathbb{H}^2(\Omega))}.
\end{equation} 
From \eqref{DS8}, \eqref{DS10} we have
\begin{align*}
\|e(t_k)\|_{\mathbb{H}^1(\Omega)}&\le \|e_h(t_k)\|_{\mathbb{H}^1(\Omega)}+\|e_p(t_k)\|_{\mathbb{H}^1(\Omega)}\\
&\le C(h\|y\|_{C([0,T];\mathbb{H}^2(\Omega))}+\tau\|y'\|_{L^2(0,T;\mathbb{H}^1(\Omega))}),
\end{align*}
for every $k=1,\ldots,N_\tau$. This estimate combining with Lemma \ref{LM2}, \eqref{DS9} imply \eqref{DS3}. Finally, we prove \eqref{DS3.1}. Assume that $t\in (t_{n-1},t_n)$ for some $n\in \{1,\ldots,N_\tau\}.$ Then 
$$\|y(t)-\hat{y}_\sigma(t)\|_{\mathbb{H}^1(\Omega)}\le \|y(t)-y(t_n)\|_{\mathbb{H}^1(\Omega)}+\|y(t_n)-\hat{y}_\sigma(t_n)\|_{\mathbb{H}^1(\Omega)}.$$
The first term can be estimated as follows
\begin{align*}
\|y(t)-y(t_n)\|_{\mathbb{H}^1(\Omega)}&=\|\int_t^{t_n}y'(s)ds\|_{\mathbb{H}^1(\Omega)}\\
&\le \int_t^{t_n}\|y'(s)\|_{\mathbb{H}^1(\Omega)}ds\le \sqrt{\tau}\|y'\|_{L^2(0,T;\mathbb{H}^1(\Omega))}.
\end{align*}
This combining with \eqref{DS3} implies \eqref{DS3.1}.

\end{proof}
\begin{theorem}
For every given $u\in \mathbb{L}^2(Q)$, \eqref{DSE1} has a unique solution $y_\sigma\in V_\sigma$. Moreover, there exists a constant $C>0$ independent of $\sigma$ such that 
\begin{multline}
\label{DS11}
\max_{1\le n\le N_\tau}\|y(t_n)-y_\sigma(t_n)\|_{\mathbb{H}^1(\Omega)} + 
\|y-y_\sigma\|_{L^2(0,T;\mathbb{H}^1(\Omega))}\\
\le C(\tau\|y'\|_{L^2(0,T;\mathbb{H}^1(\Omega))}+h\|y\|_{C([0,T];\mathbb{H}^2(\Omega))}+h\|p\|_{L^2(0,T;H^1(\Omega))}),
\end{multline}
\begin{multline}
\label{DS12}
|y-y_\sigma\|_{L^\infty(0,T;\mathbb{H}^1(\Omega))}\le
 C\{(\tau+\sqrt{\tau})\|y'\|_{L^2(0,T;\mathbb{H}^1(\Omega))}\\
 +h\|y\|_{C([0,T];\mathbb{H}^2(\Omega))}+h\|p\|_{L^2(0,T;H^1(\Omega))}\},
\end{multline}
where $y\in W^{1,2}(0,T;D(A))$ is the unique solution of \eqref{ST2}. 
\end{theorem}
\begin{proof}
Set $\varepsilon=y-y_\sigma,\,e=y-\hat{y}_\sigma,\,e_\sigma=\hat{y}_\sigma-y_\sigma,$ where $\hat{y}_\sigma$ is the solution of \eqref{DS2}. Replacing $y_{n,h}$ by $\hat{y}_{n,h}-e_{n,h}$ in \eqref{DSE1} gives
\begin{multline*}
\left(\dfrac{\hat{y}_{n,h}-\hat{y}_{n-1,h}}{\tau_n},w_h\right)+\nu a(\hat{y}_{n,h},w_h)+\alpha^2 a\left(\dfrac{\hat{y}_{n,h}-\hat{y}_{n-1,h}}{\tau_n},w_h\right)\\
-\left(\dfrac{e_{n,h}-e_{n-1,h}}{\tau_n},w_h\right)-\nu a(e_{n,h},w_h)-\alpha^2 a\left(\dfrac{e_{n,h}-e_{n-1,h}}{\tau_n},w_h\right)\\
+c(y_{n,h},y_{n,h},w_h)=\dfrac{1}{\tau_n}\int_{t_{n-1}}^{t_n}(u(t),w_h)dt.
\end{multline*}
Using \eqref{DS2} we have
\begin{multline*}
\left(\dfrac{e_{n,h}-e_{n-1,h}}{\tau_n},w_h\right)+\nu a(e_{n,h},w_h)+\alpha^2 a\left(\dfrac{e_{n,h}-e_{n-1,h}}{\tau_n},w_h\right)=\\
\dfrac{1}{\tau_n}\int_{t_{n-1}}^{t_n}\{\nu a(y(t),w_h)+\alpha^2a(y'(t),w_h)+(y'(t),w_h)-(u(t),w_h)\}dt\\
+c(y_{n,h},y_{n,h},w_h).
\end{multline*}
Combining with \eqref{ST1} we get
\begin{multline*}
\left(\dfrac{e_{n,h}-e_{n-1,h}}{\tau_n},w_h\right)+\nu a(e_{n,h},w_h)+\alpha^2 a\left(\dfrac{e_{n,h}-e_{n-1,h}}{\tau_n},w_h\right)=\\
c(y_{n,h},y_{n,h},w_h)-\dfrac{1}{\tau_n}\int_{t_{n-1}}^{t_n}c(y(t),y(t),w_h)dt-\dfrac{1}{\tau_n}\int_{t_{n-1}}^{t_n}(p(t),{\rm div}\, w_h)dt.
\end{multline*}
Therefore,
\begin{multline*}
\left(e_{n,h}-e_{n-1,h},w_h\right)+\nu \int_{t_{n-1}}^{t_n}a(e_{n,h},w_h)dt+\alpha^2 a\left(e_{n,h}-e_{n-1,h},w_h\right)=\\
\int_{t_{n-1}}^{t_n}\{c(y_{n,h},y_{n,h},w_h)-c(y(t),y(t),w_h)\}dt-\int_{t_{n-1}}^{t_n}(p(t),{\rm div}\, w_h)dt.
\end{multline*}
Setting $w_h=e_{n,h}$ we obtain
\begin{align}
&\dfrac{1}{2}|e_{n,h}|^2-\dfrac{1}{2}|e_{n-1,h}|^2+\dfrac{1}{2}|e_{n,h}-e_{n-1,h}|^2+\nu\int_{t_{n-1}}^{t_n}|\nabla e_{n,h}|^2dt\notag\\
&+\dfrac{\alpha^2}{2}|\nabla e_{n,h}|^2-\dfrac{\alpha^2}{2}|\nabla e_{n-1,h}|^2+\dfrac{\alpha^2}{2}|\nabla(e_{n,h}-e_{n-1,h})|^2\notag\\
=&\int_{t_{n-1}}^{t_n}\{c(y_{n,h},y_{n,h},e_{n,h})-c(y(t),y(t),e_{n,h})\}dt-\int_{t_{n-1}}^{t_n}(p(t),{\rm div}\, e_{n,h})dt\notag\\
=&\int_{t_{n-1}}^{t_n}\{c(y_{n,h},y_{n,h},e_{n,h})-c(y(t),y(t),e_{n,h})\}dt-\int_{t_{n-1}}^{t_n}(p(t)-R_hp(t),{\rm div}\, e_{n,h})dt.\label{DS13}
\end{align}
After some algebraic computations we get for every $t\in (t_{n-1},t_n)$ that
\begin{align*}
&c(y(t),y(t),e_{n,h})-c(y_{n,h},y_{n,h},e_{n,h})\\
=&c(e(t),y(t),e_{n,h})+c(\hat{y}_{n,h},e(t),e_{n,h})+c(e_{n,h},\hat{y}_{n,h},e_{n,h})+c(y_{n,h},e_{n,h},e_{n,h})\\
=&c(e(t),y(t),e_{n,h})+c(\hat{y}_{n,h},e(t),e_{n,h})+c(e_{n,h},\hat{y}_{n,h},e_{n,h}).
\end{align*}
Hence, we get from \eqref{DS13} that
\begin{align}
&\dfrac{1}{2}|e_{n,h}|^2-\dfrac{1}{2}|e_{n-1,h}|^2+\dfrac{1}{2}|e_{n,h}-e_{n-1,h}|^2+\nu\int_{t_{n-1}}^{t_n}|\nabla e_{n,h}|^2dt\notag\\
&+\dfrac{\alpha^2}{2}|\nabla e_{n,h}|^2-\dfrac{\alpha^2}{2}|\nabla e_{n-1,h}|^2+\dfrac{\alpha^2}{2}|\nabla(e_{n,h}-e_{n-1,h})|^2\notag\\
\le\; & \int_{t_{n-1}}^{t_n}\{|c(e(t),y(t),e_{n,h})|+|c(\hat{y}_{n,h},e(t),e_{n,h})|+|c(e_{n,h},\hat{y}_{n,h},e_{n,h})|\}dt\notag\\
&+\int_{t_{n-1}}^{t_n}|(p(t)-R_hp(t),{\rm div}\,e_{n,h})|dt.\label{DS14}
\end{align}
Since $y\in W^{1,2}(0,T;V)$ and $\{\hat{y}_\sigma\}_\sigma$ is bounded in $L^\infty(0,T;\mathbb{H}^1(\Omega))$, we have
 \begin{align*}
 \int_{t_{n-1}}^{t_n}|c(e(t),y(t),e_{n,h})|dt \le & C\int_{t_{n-1}}^{t_n}|\nabla e(t)||\nabla e_{n,h}|dt\\
 \le &C\int_{t_{n-1}}^{t_n}|\nabla e(t)|^2dt+\dfrac{\nu}{8}\int_{t_{n-1}}^{t_n}|\nabla e_{n,h}|^2dt,\\
  \int_{t_{n-1}}^{t_n}|c(\hat{y}_{n,h},e(t),e_{n,h})|dt \le & C\int_{t_{n-1}}^{t_n}|\nabla e(t)||\nabla e_{n,h}|dt\\
 \le & C\int_{t_{n-1}}^{t_n}|\nabla e(t)|^2dt+\dfrac{\nu}{8}\int_{t_{n-1}}^{t_n}|\nabla e_{n,h}|^2dt,\\
  \int_{t_{n-1}}^{t_n}|c(e_{n,h},\hat{y}_{n,h},e_{n,h})|dt\le & C\int_{t_{n-1}}^{t_n}|e_{n,h}|^{1/4}|\nabla e_{n,h}|^{7/4}dt\\
 \le & C\tau_n|e_{n,h}|^2+\dfrac{\nu}{8}\int_{t_{n-1}}^{t_n}|\nabla e_{n,h}|^2dt,\\
  \int_{t_{n-1}}^{t_n}|(p(t)-R_hp(t),{\rm div}\, e_{n,h})|dt\le & C\int_{t_{n-1}}^{t_n}|p(t)-R_hp(t)|^2dt+\dfrac{\nu}{8}|\nabla e_{n,h}|^2dt.
  \end{align*}
 Putting all these estimates in \eqref{DS14} we obtain
 \begin{multline*}
 (1-C\tau_{n})|e_{n,h}|^2+|e_{n,h}-e_{n-1,h}|^2+\nu\int_{t_{n-1}}^{t_n}|\nabla e_{n,h}|^2dt\\
+\alpha^2|\nabla e_{n,h}|^2+\alpha^2|\nabla(e_{n,h}-e_{n-1,h})|^2\\
\le |e_{n-1,h}|^2+\alpha^2|\nabla e_{n-1,h}|^2+ C\int_{t_{n-1}}^{t_n}|\nabla e(t)|^2dt+C\int_{t_{n-1}}^{t_n}|p(t)-R_hp(t)|^2dt.
 \end{multline*}
 Adding these inequalities for $n=1,2,\ldots,k$, and noticing that $e_{0,h}=0$, we get
 \begin{multline*}
 (1-C\tau)|e_{k,h}|^2+\alpha^2|\nabla e_{k,h}|^2+\nu\int_0^{t_k}|\nabla e_\sigma(t)|^2dt\\
 \le \sum_{n=1}^{k-1}C\tau_n|e_{n,h}|^2+ C\int_{0}^{t_k}|\nabla e(t)|^2dt+C\int_{0}^{t_k}|p(t)-R_hp(t)|^2dt.
 \end{multline*}
 Using the discrete Gronwall inequality we have, for every $k=1,2,\ldots,N_\tau$,
 \begin{align*}
 \|e_{k,h}\|^2_{\mathbb{H}^1(\Omega)}+\int_0^{t_k}|\nabla e_\sigma(t)|^2dt
 \le  &C\int_0^T|\nabla e(t)|^2dt+C\int_0^T|p(t)-R_hp(t)|^2dt\\
 \le  &C\|e\|^2_{L^2(0,T;\mathbb{H}^1(\Omega))}+Ch^2\|p\|^2_{L^2(0,T;H^1(\Omega))}.
 \end{align*}
 This together with \eqref{DS3} imply \eqref{DS11}. By a similar argument as in the proof of \eqref{DS3.1} we get \eqref{DS12} from \eqref{DS11}.
 
 To finish the proof, we have to prove the uniqueness of a solution to \eqref{DSE1}. Assume that $y_\sigma^1, y_\sigma^2\in V_\sigma$ are two solutions of \eqref{DSE1}. Setting $y_\sigma=y_\sigma^2-y_\sigma^1$, then we need to prove that $y_\sigma=0$. Subtracting \eqref{DSE1} for $y_\sigma^2-y_\sigma^1$ and taking $w_h=y_{n,h}$ we get
 \begin{multline}
 \label{DS15}
 \left(\dfrac{y_{n,h}-y_{n-1,h}}{\tau_n},y_{n,h}\right)+\nu a(y_{n,h},y_{n,h})+\alpha^2 a\left(\dfrac{y_{n,h}-y_{n-1,h}}{\tau_n},y_{n,h}\right)\\
=-c(y_{n,h},y_{n,h}^1,y_{n,h}).
 \end{multline}
 Here, we have used the fact that 
 $$c(y^1_{n,h},y^1_{n,h},y_{n,h})-c(y^2_{n,h},y^2_{n,h},y_{n,h})=-c(y_{n,h},y_{n,h}^1,y_{n,h}).$$ 
It follows from \eqref{DS12} that $\|y^1_{\sigma}\|_{L^\infty(0,T;\mathbb{H}^1(\Omega))}\le C$, where $C$ is a constant depending only on $y,\sigma$. Hence, we can get from \eqref{DS15} that
\begin{multline*}
\dfrac{1}{2}|y_{n,h}|^2-\dfrac{1}{2}|y_{n-1,h}|^2+\dfrac{1}{2}|y_{n,h}-y_{n-1,h}|^2+\nu\tau_n|\nabla y_{n,h}|^2\\
+\dfrac{\alpha^2}{2}|\nabla y_{n,h}|^2-\dfrac{\alpha^2}{2}|\nabla y_{n-1,h}|^2+\dfrac{\alpha^2}{2}|\nabla (y_{n,h}-y_{n-1,h})|^2\\
\le C\tau_n|y_{n,h}|^{1/4}|\nabla y_{n,h}|^{7/4}\le C\tau_n|y_{n,h}|^2+\dfrac{\nu}{2}\tau_n|\nabla y_{n,h}|^2.
\end{multline*}
 Hence,
 \begin{equation*}
 (1-C\tau_n)|y_{n,h}|^2+\alpha^2|\nabla y_{n,h}|^2\le |y_{n-1,h}|^2+\alpha^2|\nabla y_{n-1,h}|^2.
 \end{equation*}
 Adding these estimates for $n=1,2,\ldots,k$, and noticing that $y_{0,h}=0$, we get
 $$(1-C\tau)|y_{k,h}|^2+\alpha^2|\nabla y_{k,h}|^2\le \sum_{n=1}^{k-1}C\tau_n|y_{n,h}|^2.$$
 Using once again the discrete Gronwall inequality we conclude that $y_\sigma=0.$
 \end{proof}
 \begin{remark}{\rm 
\label{rm1} 
 According to the proof above, the constants $C$ in \eqref{DS11} and \eqref{DS12} are dependent on $\|y\|_{W^{1,2}(0,T;V)}$, $\|\hat{y}_\sigma\|_{L^\infty(0,T;\mathbb{H}^1(\Omega))}$. However, by using Theorem \ref{tontaiNSV} and \eqref{DS7.1} we see that if $\|u\|_{L^2(0,T;\mathbb{L}^2(\Omega))}\le M$ then these constants depend only on $M$, not on $y,u$.}
 \end{remark}
 \begin{corollary}
 \label{Cor4.1}
 Assume that $\max\{\|u\|_{L^2(0,T;\mathbb{L}^2(\Omega))},\|v\|_{L^2(0,T;\mathbb{L}^2(\Omega))}\}\le M$. Denote by $y_u\in W^{1,2}(0,T;D(A))$ the unique solution of \eqref{ST2} and by $y_\sigma(v)\in V_\sigma$ the unique solution of \eqref{DSE1} corresponding to the control $v$. Then there exists a constant $C_M>0$ such that
 \begin{equation}
 \label{DS15.2}
 \|y_u-y_\sigma(v)\|_{L^\infty(0,T;\mathbb{H}^1(\Omega))}\le C_M\{h+\tau+\sqrt{\tau}+\|u-v\|_{L^2(0,T;\mathbb{L}^2(\Omega))}\}.
 \end{equation}
 Moreover, if $u_\sigma\in U_\sigma$ and $u_\sigma\rightharpoonup u$ in $L^2(0,T;\mathbb{L}^2(\Omega_h))$ as $\sigma\to 0$ then
 \begin{equation}
 \label{DS15.3}
 \|y_u-y_\sigma(u_\sigma)\|_{L^\infty(0,T;\mathbb{H}^1(\Omega))}\to 0 \text{ as } \sigma\to 0,
 \end{equation}
 \begin{equation}
 \label{DS15.3.0}
 \|y_u(T)-y_\sigma(u_\sigma)(T)\|_{\mathbb{H}^1(\Omega)}\to 0 \text{ as } \sigma\to 0.
 \end{equation}
 \end{corollary}
 \begin{proof}
 We have
 \begin{align*}
 \|y_u-y_\sigma(v)\|_{L^\infty(0,T;\mathbb{H}^1(\Omega))}&\le \|y_u-y_v\|_{L^\infty(0,T;\mathbb{H}^1(\Omega))}+\|y_v-y_\sigma(v)\|_{L^\infty(0,T;\mathbb{H}^1(\Omega))}\\
 &= \|G(u)-G(v)\|_{L^\infty(0,T;\mathbb{H}^1(\Omega))} + \|y_v-y_\sigma(v)\|_{L^\infty(0,T;\mathbb{H}^1(\Omega))}.
 \end{align*}
 From \eqref{DS12}, \eqref{Pre1.0} and Remark \ref{rm1} we have
\begin{equation}
\label{DS15.3.1}
\|y_v-y_\sigma(v)\|_{L^\infty(0,T;\mathbb{H}^1(\Omega))}\le C_M(h+\tau+\sqrt{\tau}).
\end{equation} 
In addition, the control-to-state mapping $G$ is of class $C^2$, so we can use mean value theorem, \eqref{DGE2.1}, and \eqref{ST2.1}  to get \eqref{DS15.2}.

Next, we are going to prove \eqref{DS15.3}. We have
\begin{equation}
\label{DS15.4}
\|y_u-y_\sigma(u_\sigma)\|_{L^\infty(0,T;\mathbb{H}^1(\Omega))}\le \|y_u-y_{u_\sigma}\|_{L^\infty(0,T;\mathbb{H}^1(\Omega))} + \|y_{u_\sigma}-y_\sigma(u_\sigma)\|_{L^\infty(0,T;\mathbb{H}^1(\Omega))}.
\end{equation}
Since $u_\sigma\rightharpoonup u$ in $L^2(0,T;\mathbb{L}^2(\Omega_h))$ as $\sigma\to 0$ we get the boundedness of the sequence $\{u_\sigma\}_\sigma$ in the space $L^2(0,T;\mathbb{L}^2(\Omega_h))$. Then, from \eqref{DS15.3.1} we get that the second term in the right-hand side of \eqref{DS15.4} tends to $0$ as $\sigma\to 0$. By Theorem \ref{THR2.3}, we can extract from sequence $\{y_{u_\sigma}\}_\sigma$ a subsequence denoted in the same way such that
\begin{equation}
\label{DS15.5}
y_{u_\sigma}\rightharpoonup y_u \text{ in } W^{1,2}(0,T;D(A)).
\end{equation}
Since the embedding $W^{1,2}(0,T;D(A))\hookrightarrow L^\infty(0,T;\mathbb{H}^1(\Omega))$ is compact, we get
$$\|y_u-y_{u_\sigma}\|_{L^\infty(0,T;\mathbb{H}^1(\Omega))}\to 0 \text{ as } \sigma\to 0,$$
and \eqref{DS15.3} is proved. It follows from \eqref{DS15.5} that $y_{u_\sigma}(T)\rightharpoonup y_u(T)$ in $D(A)$. On the other hand, since $D(A)$ is compactly embedded in $\mathbb{H}^1(\Omega)$, then we have 
$$\|y_{u_\sigma}(T)-y_u(T)\|_{\mathbb{H}^1(\Omega)}\to 0.$$
This together with \eqref{DS11} imply \eqref{DS15.3.0}. The proof is complete.
 \end{proof}
 \begin{theorem}
 \label{THRDiffG}
 The mapping $G_\alpha: L^2(0,T;\mathbb{L}^2(\Omega))\to V_\sigma$ defined by $G_\sigma(u)=y_\sigma(u)$, the solution of \eqref{DSE1}, is of class $C^\infty$. Moreover, $z_\sigma(v)=G'_\sigma(u)v$ is the unique solution of the following problem
 \begin{equation}
 \label{DifDSE}
 \begin{cases}
 \left(\dfrac{z_{n,h}-z_{n-1,h}}{\tau_n},w_h\right) + \nu a(z_{n,h},w_h) +\alpha^2 a\left(\dfrac{z_{n,h}-z_{n-1,h}}{\tau_n},w_h\right)\\
 \hspace{1cm}+c(z_{n,h},y_{n,h},w_h)+c(y_{n,h},z_{n,h},w_h)=\dfrac{1}{\tau_n}\int_{t_{n-1}}^{t_n}(v(t),w_h)dt,\\
 \hspace{6cm}\forall w_h\in V_h,\;\forall n=1,2,\ldots,N_\tau,\\
  z_{0,h}=0,
 \end{cases}
 \end{equation}
 where we have set $y_\sigma=y_\sigma(u).$
 \end{theorem}
 \begin{proof}
 We consider the mapping $F_\sigma: V_\sigma\times L^2(0,T;\mathbb{L}^2(\Omega))\to {V'_h}^{N_\tau}\times V_h$ defined by $F_\sigma(y_\sigma,u)=(g_1,\ldots,g_{N_{\tau}},y_{0,h}-P_hy_0)$, where
 \begin{align*}
 \langle g_n,w_h\rangle=&(y_{n,h}-y_{n-1,h},w_h)+\nu\tau_na(y_{n,h},w_h)+\alpha^2a(y_{n,h}-y_{n-1,h},w_h)\\
 &+\tau_nc(y_{n,h},y_{n,h},w_h)
 -\int_{t_{n-1}}^{t_n}(u(t),w_h)dt,\;\;\;\forall w_h\in V_h,\;\forall n=1,2,\ldots,N_\tau.
 \end{align*}
 We can easily check that $F_\sigma$ is of class $C^\infty$ and for an arbitrary $e_\sigma\in V_\sigma$ we have
 \begin{equation}
 \label{DS15.1}
  \dfrac{\partial F_\sigma}{\partial y_\sigma}(y_\sigma,u)e_\sigma
 = (f_1,\ldots,f_{N_\tau},e_{0,h}),
 \end{equation}
 where $f_n\in V'_h$ is defined by
 \begin{multline}
 \label{DS16}
 \langle f_n,w_h\rangle =(e_{n,h}-e_{n-1,h},w_h)+\nu\tau_na(e_{n,h},w_h)
 +\alpha^2a(e_{n,h}-e_{n-1,h},w_h)\\
 +\tau_nc(y_{n,h},e_{n,h},w_h)+\tau_nc(e_{n,h},y_{n,h},w_h),\quad\forall w_h\in V_h,\;\forall n= 1,2, \ldots,N_\tau.
 \end{multline}
 On the other hand, $F_\sigma(G_\sigma(u),u)=F_\sigma(y_\sigma(u),u)=0$ for every $u\in L^2(0,T;\mathbb{L}^2(\Omega))$. The proof is a consequence of the implicit function; we need to prove that $\dfrac{\partial F_\sigma}{\partial y_\sigma}(y_\sigma,u):V_\sigma\to {V'_h}^{N_\tau}\times V_h$ is an isomorphism for every $(y_\sigma,u)\in V_\sigma\times L^2(0,T;\mathbb{L}^2(\Omega))$. Since $V_\sigma$ and ${V'_h}^{N_\tau}\times V_h$ are spaces with the same finite dimension, we only need to prove that $\dfrac{\partial F_\sigma}{\partial y_\sigma}(y_\sigma,u)$ is injective. Suppose that $\dfrac{\partial F_\sigma}{\partial y_\sigma}(y_\sigma,u)e_\sigma=0$ for some $e_\sigma\in V_\sigma$, then \eqref{DS15.1} implies that $e_{0,h}=0$. Using \eqref{DS16} with $w_h=e_{n,h}$ we have
 \begin{multline*}
 (e_{n,h}-e_{n-1,h},e_{n,h})+\nu\tau_na(e_{n,h},e_{n,h})
 +\alpha^2a(e_{n,h}-e_{n-1,h},e_{n,h})\\+\tau_nc(e_{n,h},y_{n,h},e_{n,h})=0.
 \end{multline*}
 Hence,
 \begin{multline*}
 \dfrac{1}{2}|e_{n,h}|^2-\dfrac{1}{2}|e_{n-1,h}|^2+\dfrac{1}{2}|e_{n,h}-e_{n-1,h}|^2+\nu\tau_n|\nabla e_{n,h}|^2
 +\dfrac{\alpha^2}{2}|\nabla e_{n,h}|^2\\
 -\dfrac{\alpha^2}{2}|\nabla e_{n-1,h}|^2+\dfrac{\alpha^2}{2}|\nabla (e_{n,h}-e_{n-1,h})|^2
 \le C\tau_n|e_{n,h}|^{1/4}|\nabla e_{n,h}|^{7/4}|\nabla y_{n,h}|\\
 \le \dfrac{\nu}{2}\tau_n|\nabla e_{n,h}|^2 + C\tau_n |e_{n,h}|^2.
 \end{multline*}
 Here, the constant $C$ depends on $\nu$ and $\|y_\sigma\|_{L^\infty(0,T;\mathbb{H}^1(\Omega))}$. Again, by using the discrete Gronwall inequality and noticing that $e_{0,h}=0$, we obtain $e_\sigma=0$.
 \end{proof}
\subsection{Analysis of the discrete adjoint state equations}
It follows from Theorem \ref{THRDiffG} and the chain rule that the functional $J_\sigma: L^2(0,T;\mathbb{L}^2(\Omega))\to \mathbb{R}$ is of class $C^\infty$, and its derivative is given by
\begin{multline*}
J'_\sigma(u)v=\alpha_T\int_{\Omega_h}(y_\sigma(T)-y^h_T)z_\sigma(T)dx + \alpha_Q\int_0^T\int_{\Omega_h}(y_\sigma-y_Q)z_\sigma dxdt\\
+\gamma\int_0^T\int_{\Omega_h}uvdxdt,
\end{multline*}
where $y_\sigma=y_\sigma(u)=G_\sigma(u)$ and $z_\sigma=G'_\sigma(u)v$ is the solution of \eqref{DifDSE}. 

To study the discrete adjoint state equation, we are going to introduce the space $V_\sigma^r$. For each given sequence $(\lambda_{1,h},\lambda_{2,h}, \ldots,\lambda_{N_\tau+1,h})\in V_h^{N_\tau+1}$, we define a function $\lambda_\sigma:[0,T]\to V_h$ by
\begin{equation*}
\begin{cases}
\lambda_\sigma(t_n)=\lambda_{n+1,h},\;n=0,1,\ldots,N_\tau,\\
\lambda_\sigma(t)=\lambda_{n,h}, \;\forall t\in (t_{n-1},t_n).
\end{cases}
\end{equation*}
We denote by $V_\sigma^r$ the set of all functions that are defined by this way.

Now, we consider the discrete adjoint state equation: 

{\it Find} $\lambda_\sigma\in V_\sigma^r$ {\it such that}
\begin{equation}
\label{ADE1}
\begin{cases}
\left(\dfrac{\lambda_{n,h}-\lambda_{n+1,h}}{\tau_n},w_h\right) + \nu a(\lambda_{n,h},w_h)+\alpha^2a\left(\dfrac{\lambda_{n,h}-\lambda_{n+1,h}}{\tau_n},w_h\right)\\
+c(w_h,y_{n,h},\lambda_{n,h})+c(y_{n,h},w_h,\lambda_{n,h})=\dfrac{\alpha_Q}{\tau_n}\int_{t_{n-1}}^{t_n}(y_{n,h}-y_Q(t),w_h)dt,\\
\hspace{6cm}\forall w_h\in V_h,\;\forall n=1,2,\ldots,N_\tau,\\
(\lambda_{N_\tau+1,h},w_h)+\alpha^2a(\lambda_{N_\tau+1,h},w_h)=\alpha_T(y_{N_\tau,h}-y_T^h,w_h),\quad\forall w_h\in V_h.
\end{cases}
\end{equation}
In this system, first we compute $\lambda_{N_\tau+1,h}$ from the last equation in \eqref{ADE1}, then we descend in $n$ until $n=1$. We can prove the existence and uniqueness of a solution to \eqref{ADE1} by a similar way that we did for the system \eqref{DSE1}. Now, we are going to check that \eqref{ADE1} is actually the discrete adjoint state equation. Indeed, it follows from \eqref{DifDSE} and \eqref{ADE1} that
\begin{align*}
&\alpha_Q\int_0^T\int_{\Omega_h}(y_\sigma-y_Q)z_\sigma dxdt\\
&=\sum_{n=1}^{N_\tau}\alpha_Q\int_{t_{n-1}}^{t_n}(y_{n,h}-y_Q(t),z_{n,h})dt=\sum_{n=1}^{N_\tau}[(\lambda_{n,h}-\lambda_{n+1,h},z_{n,h})+\nu\tau_na(\lambda_{n,h},z_{n,h})]\\
&+\sum_{n=1}^{N_\tau}[\alpha^2a(\lambda_{n,h}-\lambda_{n+1,h},z_{n,h})+\tau_nc(z_{n,h},y_{n,h},\lambda_{n,h})+\tau_nc(y_{n,h},z_{n,h},\lambda_{n,h})]\\
&=\sum_{n=1}^{N_\tau}[(z_{n,h}-z_{n-1,h},\lambda_{n,h})+\nu\tau_na(z_{n,h},\lambda_{n,h})+\alpha^2a(z_{n,h}-z_{n-1,h},\lambda_{n,h})]\\
&+\sum_{n=1}^{N_\tau}[\tau_nc(z_{n,h},y_{n,h},\lambda_{n,h})+\tau_nc(y_{n,h},z_{n,h},\lambda_{n,h})]\\
&\hspace{6cm}-(\lambda_{N_\tau+1,h},z_{N_\tau,h})-\alpha^2a(\lambda_{N_\tau+1,h},z_{N_\tau,h})\\
&=\int_{0}^T\int_{\Omega_h}v\lambda_\sigma dxdt -\alpha_T\int_{\Omega_h}(y_{N_{\tau},h}-y_T^h)z_{N_\tau,h}dx\\
&=\int_{0}^T\int_{\Omega_h}v\lambda_\sigma dxdt -\alpha_T\int_{\Omega_h}(y_\sigma(T)-y_T^h)z_\sigma(T)dx.
\end{align*}
Here, we have used the fact that $z_{0,h}=0$. Hence,
\begin{equation*}
J'_\sigma(u)v=\int_0^T\int_{\Omega_h}(\lambda_\sigma + \gamma u)v dxdt.
\end{equation*}
The next theorem gives us the error estimates when approximating the adjoint state equation.
\begin{theorem}
Given $u\in L^2(0,T;\mathbb{L}^2(\Omega))$, let $(y,p)$ be the solution of \eqref{ST1}, $(\lambda,q,r)$ be the unique weak solution of \eqref{DGE3}, $y_\sigma=y_\sigma(u)$ be the associated discrete state, solution of \eqref{DSE1}, and $\lambda_\sigma$ be the associated discrete adjoint state, solution of \eqref{ADE1}. Then $\{\lambda_\sigma\}_\sigma$ is bounded in $L^\infty(0,T;\mathbb{H}^1(\Omega))$ and there exists a constant $C>0$ independent of $\sigma$ such that
\begin{equation}\label{ADE3}
\begin{aligned}
\|\lambda-\lambda_\sigma\|_{L^\infty(0,T;\mathbb{H}^1(\Omega))}\le &C\big\{\tau\|y'\|_{L^2(0,T;\mathbb{H}^1(\Omega))}+h\|y\|_{C([0,T];\mathbb{H}^2(\Omega))}\\\
&+h\|p\|_{L^2(0,T;H^1(\Omega))}+h\|r\|_{H^1(\Omega)}+h\|\lambda\|_{C([0,T];\mathbb{H}^2(\Omega))}\\
&+(\tau+\sqrt{\tau})\|\lambda'\|_{L^2(0,T;\mathbb{H}^1(\Omega))}
 +h\|q\|_{L^2(0,T;H^1(\Omega))}+h\big\}.
\end{aligned} 
\end{equation}
\end{theorem}
\begin{proof}
Define the operator $P'_\sigma: C([0,T];V)\to V_\sigma^r$ by 
\begin{equation}
\label{ADE4}
(P'_\sigma w)_{n,h}=P_hw(t_{n-1}),\;n=1,2, \ldots,N_\tau+1,
\end{equation}
where $P_h$ is given in Definition \ref{Def1}. Analogously to Lemma \ref{LM2} we have 
\begin{equation}
\label{ADE4.1}
\|w-P'_\sigma w\|_{L^2(0,T;\mathbb{H}^1(\Omega))}\le C(h\|w\|_{L^2(0,T;\mathbb{H}^2(\Omega))}+\tau\|w'\|_{L^2(0,T;\mathbb{H}^1(\Omega))}),
\end{equation}
for every $w\in W^{1,2}(0,T;D(A))$. 
Set 
$$\epsilon=\lambda-\lambda_\sigma=(\lambda-P'_\sigma\lambda)+(P'_\sigma\lambda-\lambda_\sigma)= \psi +\epsilon_\sigma.$$
From \eqref{ADE4} we have
$$\psi(t_n)=\lambda(t_n)- (P'_\sigma\lambda)(t_n)=\lambda(t_n)-P_h\lambda(t_n),\quad n=0,1,\ldots,N_\tau.$$
Also we have $\epsilon_\sigma(t_n)=\epsilon_{n+1,h}$. Since $\epsilon(t_n)=\lambda(t_n)-\lambda_{n+1,h}$, it follows from \eqref{DGE3} and \eqref{ADE1} that
\begin{multline*}
(\epsilon(t_{n-1})-\epsilon(t_n),w_h)+\nu\int_{t_{n-1}}^{t_n}a(\epsilon(t),w_h)dt+\alpha^2a(\epsilon(t_{n-1})-\epsilon(t_n),w_h)\\
 +\int_{t_{n-1}}^{t_n}c(w_h,y(t),\lambda(t))dt +\int_{t_{n-1}}^{t_n}c(y(t),w_h,\lambda(t))dt\\
 -\int_{t_{n-1}}^{t_n}c(w_h,y_{n,h},\lambda_{n,h})dt-\int_{t_{n-1}}^{t_n}c(y_{n,h},w_h,\lambda_{n,h})dt\\
 +\int_{t_{n-1}}^{t_n}(q(t)-R_hq(t),{\rm div}\, w_h)dt=\alpha_Q\int_{t_{n-1}}^{t_n}(y(t)-y_{n,h},w_h)dt,\quad\forall w_h\in V_h.
\end{multline*}
Here, we have used the fact that $(R_hq(t),{\rm div}\, w_h)=0\;\forall w_h\in V_h$. Now, replacing $\epsilon$ by $\psi+\epsilon_\sigma$, $w_h$ by $\epsilon_{n,h}$ and taking into account that
\begin{multline*}
(\psi(t_n),w_h)+\alpha^2a(\psi(t_n),w_h)=(\lambda(t_n)-P_h\lambda(t_n),w_h)+\alpha^2a(\lambda(t_n)-P_h\lambda(t_n),w_h)\\
=(\lambda(t_n),w_h)+\alpha^2a(\lambda(t_n),w_h)-[(P_h\lambda(t_n),w_h)+\alpha^2a(P_h\lambda(t_n),w_h)]=0,\;\;\forall w_h\in V_h,
\end{multline*}
we obtain 
\begin{multline*}
(\epsilon_{n,h}-\epsilon_{n+1,h},\epsilon_{n,h}) + \alpha^2a(\epsilon_{n,h}-\epsilon_{n+1,h},\epsilon_{n,h})+\nu\int_{t_{n-1}}^{t_n}a(\epsilon_{n,h},\epsilon_{n,h})dt\\
=\alpha_Q\int_{t_{n-1}}^{t_n}(y(t)-y_{n,h},\epsilon_{n,h})dt-\nu\int_{t_{n-1}}^{t_n}a(\psi(t),\epsilon_{n,h})dt-\int_{t_{n-1}}^{t_n}c(\epsilon_{n,h},y(t),\lambda(t))dt\\
 -\int_{t_{n-1}}^{t_n}c(y(t),\epsilon_{n,h},\lambda(t))dt
 +\int_{t_{n-1}}^{t_n}c(\epsilon_{n,h},y_{n,h},\lambda_{n,h})dt+\int_{t_{n-1}}^{t_n}c(y_{n,h},\epsilon_{n,h},\lambda_{n,h})dt\\
 -\int_{t_{n-1}}^{t_n}(q(t)-R_hq(t),{\rm div}\, \epsilon_{n,h})dt.
\end{multline*}
Hence, 
\begin{align}
&\dfrac{1}{2}|\epsilon_{n,h}|^2-\dfrac{1}{2}|\epsilon_{n+1,h}|^2
+\dfrac{1}{2}|\epsilon_{n,h}-\epsilon_{n+1,h}|^2+\nu \int_{t_{n-1}}^{t_n}|\nabla \epsilon_{n,h}|^2dt+\dfrac{\alpha^2}{2}|\nabla\epsilon_{n,h}|^2\notag\\
&-\dfrac{\alpha^2}{2}|\nabla\epsilon_{n+1,h}|^2+\dfrac{\alpha^2}{2}|\nabla(\epsilon_{n,h}-\epsilon_{n+1,h})|^2\le \alpha_Q\int_{t_{n-1}}^{t_n}|y(t)-y_\sigma(t)||\epsilon_{n,h}|dt\notag\\
&+\nu\int_{t_{n-1}}^{t_n}|\nabla \psi(t)||\nabla\epsilon_{n,h}|dt
+\int_{t_{n-1}}^{t_n}|q(t)-R_hq(t)||{\rm div}\,\epsilon_{n,h}|dt\notag\\
&+\int_{t_{n-1}}^{t_n}|c(y_{n,h},\epsilon_{n,h},\lambda_{n,h})-c(y(t),\epsilon_{n,h},\lambda(t))|dt\notag\\
&+\int_{t_{n-1}}^{t_n}|c(\epsilon_{n,h},y_{n,h},\lambda_{n,h})-c(\epsilon_{n,h},y(t),\lambda(t))|dt.\label{ADE5}
\end{align}
The right-hand side of \eqref{ADE5} can be estimated as follows
\begin{multline*}
\alpha_Q\int_{t_{n-1}}^{t_n}|y(t)-y_\sigma(t)||\epsilon_{n,h}|dt
+\nu\int_{t_{n-1}}^{t_n}|\nabla \psi(t)||\nabla\epsilon_{n,h}|dt
\le \dfrac{\tau_n}{2}|\epsilon_{n,h}|^2\\
+C\int_{t_{n-1}}^{t_n}|y(t)-y_\sigma(t)|^2dt+\dfrac{\nu}{8}\int_{t_{n-1}}^{t_n}|\nabla \epsilon_{n,h}|^2dt+C\int_{t_{n-1}}^{t_n}|\nabla\psi(t)|^2dt,
\end{multline*}
\begin{multline*}
\int_{t_{n-1}}^{t_n}|q(t)-R_hq(t)||{\rm div}\,\epsilon_{n,h}|dt\le \dfrac{\nu}{8}\int_{t_{n-1}}^{t_n}|\nabla\epsilon_{n,h}|^2dt+C\int_{t_{n-1}}^{t_n}|q(t)-R_hq(t)|^2dt,
\end{multline*}
\begin{align}
&\;\;\;\left|\int_{t_{n-1}}^{t_n}|c(y_{n,h},\epsilon_{n,h},\lambda_{n,h})-c(y(t),\epsilon_{n,h},\lambda(t))|dt\right|\notag\\
&=\left|\int_{t_{n-1}}^{t_n}[c(y_\sigma(t)-y(t),\epsilon_{n,h},\lambda(t))-c(y_\sigma(t),\epsilon_{n,h},\epsilon(t))]dt\right|\notag\\
&=\left|\int_{t_{n-1}}^{t_n}[c(y_\sigma(t)-y(t),\epsilon_{n,h},\lambda(t))-c(y_\sigma(t),\epsilon_{n,h},\psi(t))]dt\right|\notag\\
&\le C\|\lambda\|_{L^\infty(0,T;\mathbb{H}^1(\Omega))}\int_{t_{n-1}}^{t_n}\|y(t)-y_\sigma(t)\|_{\mathbb{H}^1(\Omega)}|\nabla \epsilon_{n,h}|dt\notag\\
&\;\;\;+C\|y_\sigma\|_{L^\infty(0,T;\mathbb{H}^1(\Omega))}\int_{t_{n-1}}^{t_n}\|\psi(t)\|_{\mathbb{H}^1(\Omega)}|\nabla \epsilon_{n,h}|dt\notag\\
&\le C\int_{t_{n-1}}^{t_n}\big[\|y(t)-y_\sigma(t)\|^2_{\mathbb{H}^1(\Omega)}+\|\psi(t)\|_{\mathbb{H}^1(\Omega)}^2\big]dt+\dfrac{\nu}{4}\int_{t_{n-1}}^{t_n}|\nabla \epsilon_{n,h}|^2dt.\label{ADE5.1}
\end{align}
For the last term in the right-hand side of \eqref{ADE5}, we first see that 
\begin{equation}
\begin{aligned}
\label{ADE6}
 &c(\epsilon_{n,h},y_{n,h},\lambda_{n,h})-c(\epsilon_{n,h},y(t),\lambda(t))\\
= &-[c(\epsilon_{n,h},y(t)-y_\sigma(t),\lambda(t))+c(\epsilon_{n,h},y_\sigma(t),\psi(t))+c(\epsilon_{n,h},y_\sigma(t),\epsilon_{n,h})].
\end{aligned}
\end{equation}
The first two terms in \eqref{ADE6} can be treated analogously as in \eqref{ADE5.1}. For the last we have
\begin{align}
\left|\int_{t_{n-1}}^{t_n}c(\epsilon_{n,h},y_\sigma(t),\epsilon_{n,h})dt\right|&\le C\int_{t_{n-1}}^{t_n}|\epsilon_{n,h}|^{1/4}|\nabla\epsilon_{n,h}|^{7/4}\|y_\sigma(t)\|_{\mathbb{H}^1(\Omega)}dt\notag\\
&\le C\|y_\sigma\|_{L^\infty(0,T;\mathbb{H}^1(\Omega))}\int_{t_{n-1}}^{t_n}|\epsilon_{n,h}|^{1/4}|\nabla\epsilon_{n,h}|^{7/4}dt\notag\\
&\le \dfrac{\nu}{8}\int_{t_{n-1}}^{t_n}|\nabla\epsilon_{n,h}|^2dt+C\tau_n|\epsilon_{n,h}|^2.\label{ADE7}
\end{align}
From \eqref{ADE5}-\eqref{ADE7} we get that
\begin{multline}
\label{ADE8}
(1-C\tau_n)|\epsilon_{n,h}|^2+|\epsilon_{n,h}-\epsilon_{n+1,h}|^2+\alpha^2|\nabla\epsilon_{n,h}|^2+\alpha^2|\nabla(\epsilon_{n,h}-\epsilon_{n+1,h})|^2\\
+\dfrac{\nu}{4}\int_{t_{n-1}}^{t_n}|\nabla\epsilon_{n,h}|^2dt\le |\epsilon_{n+1,h}|^2+\alpha^2|\nabla\epsilon_{n+1,h}|^2\\
+C\left\{\int_{t_{n-1}}^{t_n}\|y(t)-y_\sigma(t)\|^2_{\mathbb{H}^1(\Omega)}dt+\int_{t_{n-1}}^{t_n}\|\psi(t)\|^2_{\mathbb{H}^1(\Omega)}dt+\int_{t_{n-1}}^{t_n}|q(t)-R_hq(t)|^2dt\right\}.
\end{multline}
Set $\tilde{\epsilon}_{k,h}:=\epsilon_{N_\tau+1-k,h},\;\tilde{\tau}_k:=\tau_{N_\tau+1-k}$, $k=0,1,\ldots,N_\tau$, then \eqref{ADE8} gives
\begin{align*}
&(1-C\tilde{\tau}_k)|\tilde{\epsilon}_{k,h}|^2+|\tilde{\epsilon}_{k,h}-\tilde{\epsilon}_{k-1,h}|^2+\alpha^2|\nabla\tilde{\epsilon}_{k,h}|^2+\alpha^2|\nabla(\tilde{\epsilon}_{k,h}-\tilde{\epsilon}_{k-1,h})|^2\\
&+\dfrac{\nu}{4}\int_{t_{N_\tau-k}}^{t_{N_\tau+1-k}}|\nabla\tilde{\epsilon}_{k,h}|^2dt\le |\tilde{\epsilon}_{k-1,h}|^2+\alpha^2|\nabla\tilde{\epsilon}_{k-1,h}|^2\\
&+C\left\{\int_{t_{N_\tau-k}}^{t_{N_\tau+1-k}}\|y(t)-y_\sigma(t)\|^2_{\mathbb{H}^1(\Omega)}dt+\int_{t_{N_\tau-k}}^{t_{N_\tau+1-k}}\|\psi(t)\|^2_{\mathbb{H}^1(\Omega)}dt\right.\\
&\left.
+\int_{t_{N_\tau-k}}^{t_{N_\tau+1-k}}|q(t)-R_hq(t)|^2dt\right\}.
\end{align*}
Adding this inequality from $k=1$ to $k=n$ we have
\begin{align*}
&(1-C\tau)|\tilde{\epsilon}_{n,h}|^2+\alpha^2|\nabla \tilde{\epsilon}_{n,h}|^2+\dfrac{\nu}{4}\int_{t_{N_\tau-n}}^T|\nabla \epsilon_\sigma(t)|^2dt\\
&\le\sum_{k=1}^{n-1}C\tilde{\tau}_k|\tilde{\epsilon}_{k,h}|^2+ |\tilde{\epsilon}_{0,h}|^2+\alpha^2|\tilde{\epsilon}_{0,h}|^2+C\left\{\int_{t_{N_\tau-n}}^{T}\|y(t)-y_\sigma(t)\|^2_{\mathbb{H}^1(\Omega)}dt\right.\\&+\int_{t_{N_\tau-n}}^{T}\|\psi(t)\|^2_{\mathbb{H}^1(\Omega)}dt
\left.
+\int_{t_{N_\tau-n}}^{T}|q(t)-R_hq(t)|^2dt\right\}.
\end{align*}
Using the discrete Gronwall inequality we have, for every $n=0,1,\ldots,N_\tau$,
\begin{multline*}
\|\tilde{\epsilon}_{n,h}\|^2_{\mathbb{H}^1(\Omega)}\le \|\tilde{\epsilon}_{0,h}\|^2_{\mathbb{H}^1(\Omega)}
+C\left\{\int_0^{T}\|y(t)-y_\sigma(t)\|^2_{\mathbb{H}^1(\Omega)}dt\right.\\+\int_0^{T}\|\psi(t)\|^2_{\mathbb{H}^1(\Omega)}dt
\left.
+\int_0^{T}|q(t)-R_hq(t)|^2dt\right\}.
\end{multline*}
Hence,
\begin{multline}
\label{ADE8.1}
\|\epsilon_{n,h}\|^2_{\mathbb{H}^1(\Omega)}\le \|P_h\lambda(T)-\lambda_{N_\tau+1,h}\|^2_{\mathbb{H}^1(\Omega)}
+C\left\{\int_0^{T}\|y(t)-y_\sigma(t)\|^2_{\mathbb{H}^1(\Omega)}dt\right.\\+\int_0^{T}\|\psi(t)\|^2_{\mathbb{H}^1(\Omega)}dt
\left.
+\int_0^{T}|q(t)-R_hq(t)|^2dt\right\},\quad\forall n=1,2,\ldots,N_\tau+1.
\end{multline}
Now, we are going to estimate $\|P_h\lambda(T)-\lambda_{N_\tau+1,h}\|_{\mathbb{H}^1(\Omega)}$. Set 
$$\varepsilon=\lambda(T)-\lambda_{N_\tau+1,h}=(\lambda(T)-P_h\lambda(T))+(P_h\lambda(T)-\lambda_{N_\tau+1,h})=\varepsilon_1+\varepsilon_2.$$
Replacing $\lambda_{N_\tau+1,h}$ by $\lambda(T)-\varepsilon$ in the last equation in \eqref{ADE1} we have
\begin{multline*}
(\varepsilon,w_h)+\alpha^2a(\varepsilon,w_h)=(\lambda(T),w_h)+\alpha^2a(\lambda(T),w_h)-\alpha_T(y_{N_\tau,h}-y_T^h,w_h).
\end{multline*}
Replacing $\varepsilon$ by $\varepsilon_1+\varepsilon_2$, $w_h$ by $\varepsilon_2$ and notice that $(\varepsilon_1,w_h)+\alpha^2a(\varepsilon_1,w_h)=0\;\forall w_h\in V_h$, we get
\begin{align*}
|\varepsilon_2|^2+\alpha^2|\nabla\varepsilon_2|^2
&=(\lambda(T),\varepsilon_2)+\alpha^2a(\lambda(T),\varepsilon_2)-\alpha_T(y_{N_\tau,h}-y_T^h,\varepsilon_2)\\
&=\alpha_T(y(T)-y_T,\varepsilon_2)-(r,{\rm div}\,\varepsilon_2)-\alpha_T(y_{N_\tau,h}-y_T^h,\varepsilon_2)\\
&=\alpha_T(y(T)-y_{N_\tau,h},\varepsilon_2)-\alpha_T(y_T-y_T^h,\varepsilon_2)-(r,{\rm div}\,\varepsilon_2)\\
&=\alpha_T(y(T)-y_{N_\tau,h},\varepsilon_2)-\alpha_T(y_T-y_T^h,\varepsilon_2)-(r-R_hr,{\rm div}\,\varepsilon_2).
\end{align*}
This implies that
$$\|\varepsilon_2\|_{\mathbb{H}^1(\Omega)}\le C(|y(T)-y_{N_\tau,h}|+|y_T-y_T^h|+|r-R_hr|).$$
Hence,
\begin{multline*}
\|\varepsilon_2\|_{\mathbb{H}^1(\Omega)}\le C(\tau\|y'\|_{L^2(0,T;\mathbb{H}^1(\Omega))}+h\|y\|_{C([0,T];\mathbb{H}^2(\Omega))}+h\|p\|_{L^2(0,T;H^1(\Omega))}\\
+h\|r\|_{H^1(\Omega)}+h).
\end{multline*}
This combining with \eqref{DS11}, \eqref{ADE4.1}, \eqref{ADE8.1} imply that
\begin{multline*}
\|\epsilon_{n,h\|_{\mathbb{H}^1(\Omega)}}\le C\big\{\tau\|y'\|_{L^2(0,T;\mathbb{H}^1(\Omega))}+h\|y\|_{C([0,T];\mathbb{H}^2(\Omega))}+h\|p\|_{L^2(0,T;H^1(\Omega))}\\
+h\|r\|_{H^1(\Omega)}+h\|\lambda\|_{L^2(0,T;\mathbb{H}^2(\Omega))}+\tau\|\lambda'\|_{L^2(0,T;\mathbb{H}^1(\Omega))} +h\|q\|_{L^2(0,T;H^1(\Omega))}+h\big\}\\
\forall n =1,2,\ldots,N_\tau+1.
\end{multline*}
We have
$$\|\psi(t_n)\|_{\mathbb{H}^1(\Omega)}=\|\lambda(t_n)-P_h\lambda(t_n)\|_{\mathbb{H}^1(\Omega)}\le Ch\|\lambda(t_n)\|_{\mathbb{H}^2(\Omega)}\le Ch\|\lambda\|_{C[0,T];\mathbb{H}^2(\Omega)},$$
for every $n=0,1,\ldots,N_\tau.$ Since $\epsilon(t_n)=\psi(t_n)-\epsilon_{n+1,h}$ we get
\begin{multline*}
\|\epsilon(t_n)\|_{\mathbb{H}^1(\Omega)}\le C\big\{\tau\|y'\|_{L^2(0,T;\mathbb{H}^1(\Omega))}+h\|y\|_{C([0,T];\mathbb{H}^2(\Omega))}+h\|p\|_{L^2(0,T;H^1(\Omega))}\\
+h\|r\|_{H^1(\Omega)}+h\|\lambda\|_{C([0,T];\mathbb{H}^2(\Omega))}+\tau\|\lambda'\|_{L^2(0,T;\mathbb{H}^1(\Omega))} +h\|q\|_{L^2(0,T;H^1(\Omega))}+h\big\}\\
\forall n =0,1,\ldots,N_\tau.
\end{multline*}
Now, assume that $t\in (t_{n-1},t_n)$, then 
$$\epsilon(t)=\lambda(t)-\lambda_{n,h}=\lambda(t)-\lambda(t_{n-1})+(\lambda(t_{n-1})-\lambda_{n,h})=\lambda(t)-\lambda(t_{n-1})+\epsilon(t_{n-1}).$$
Analogously as in the last paragraph in the proof of Lemma \ref{LM53} we have
$$\|\lambda(t)-\lambda(t_{n-1})\|_{\mathbb{H}^1(\Omega)}\le \sqrt{\tau}\|\lambda'\|_{L^2(0,T;\mathbb{H}^1(\Omega))},$$
then we get \eqref{ADE3}.
\end{proof}
 \begin{remark}{\rm 
\label{rm2} 
 According to the proof above, the constant $C$ in \eqref{ADE3} depends on $\|\lambda\|_{L^{\infty}(0,T;\mathbb{H}^1(\Omega))}$, $\|y_\sigma\|_{L^\infty(0,T;\mathbb{H}^1(\Omega))}$. However, we see that if $\|u\|_{L^2(0,T;\mathbb{L}^2(\Omega))}\le M$ then this constant depends only on $M$, not on $\lambda, y, u$.}
 \end{remark}
\begin{corollary}
Assume that $\max\{\|u\|_{L^2(0,T;\mathbb{L}^2(\Omega))},\|v\|_{L^2(0,T;\mathbb{L}^2(\Omega))}\}\le M$. Let $\lambda_u\in W^{1,2}(0,T;D(A))$ be the solution of \eqref{DGE4} and $\lambda_\sigma(v)\in V_\sigma$ be the solution of the discrete equation \eqref{ADE1} corresponding to the control $v$. Then there exists a constant $C_M>0$ such that
\begin{equation}
\label{ADE9}
\|\lambda_u-\lambda_\sigma(v)\|_{L^\infty(0,T;\mathbb{H}^1(\Omega))}\le C_M \left\{h+\tau+\sqrt{\tau}+\|u-v\|_{L^2(0,T;\mathbb{L}^2(\Omega))}\right\}.
\end{equation}
\end{corollary}
\begin{proof}
From \eqref{ADE3} we have
\begin{equation}
\label{ADE9.1}
\|\lambda_u-\lambda_\sigma(v)\|_{L^\infty(0,T;\mathbb{H}^1(\Omega))}\le \|\lambda_u-\lambda_v\|_{L^\infty(0,T;\mathbb{H}^1(\Omega)}+C(h+\tau+\sqrt{\tau}),
\end{equation}
where $C$ depends on $M$. Setting $\lambda=\lambda_u-\lambda_v$, then from \eqref{DGE5} we get
\begin{multline*}
-(\lambda_t,w)+\nu a(\lambda,w) -\alpha^2a(\lambda_t,w)=c(y_v,w,\lambda_v)+c(w,y_v,\lambda_v)\\ - c(y_u,w,\lambda_u)-c(w,y_u,\lambda_u)
+\alpha_Q(y_u-y_v,w).
\end{multline*}
Taking $w=\lambda$ and using the following identities
$$c(y_v,\lambda,\lambda_v)-c(y_u,w,\lambda_u)=c(y_v-y_u,\lambda,\lambda_v),$$
$$c(\lambda,y_v,\lambda_v)-c(\lambda,y_u,\lambda_u)=c(\lambda,y_v-y_u,\lambda_v)-c(\lambda,y_u,\lambda),$$
we obtain
\begin{multline*}
-(\lambda_t,\lambda)+\nu a(\lambda,\lambda) -\alpha^2a(\lambda_t,\lambda)=c(y_v-y_u,\lambda,\lambda_v)+c(\lambda,y_v-y_u,\lambda_v)\\
-c(\lambda,y_u,\lambda)
+\alpha_Q(y_u-y_v,\lambda).
\end{multline*}
Integrating from $t$ to $T$ then integrating by parts yields
\begin{multline}
\label{ADE10}
\dfrac{1}{2}|\lambda(t)|^2-\dfrac{1}{2}|\lambda(T)|^2+\nu\int_t^T|\nabla \lambda(s)|^2ds+\dfrac{\alpha^2}{2}|\nabla \lambda(t)|^2-\dfrac{\alpha^2}{2}|\nabla \lambda(T)|^2\\
\le C\left\{\int_t^T|\nabla y_v(s)-\nabla y_u(s)||\nabla\lambda(s)||\nabla \lambda_v(s)|ds\right.\\
+\int_t^T |\nabla\lambda(s)||\nabla y_v(s)-\nabla y_u(s)||\nabla\lambda_v(s)|ds
\left.+\int_t^T |\nabla\lambda(s)|^2|\nabla y_u(s)|ds\right\}\\
+\alpha_Q\int_t^T |y_u(s)-y_v(s)||\lambda(s)|ds.
\end{multline}
Using again \eqref{DGE5} we have
$$|\lambda(T)|^2+\alpha^2|\nabla\lambda(T)|^2=\alpha_T(y_u(T)-y_v(T),\lambda(T)).$$
Therefore,
$$|\lambda(T)|^2+\alpha^2|\nabla\lambda(T)|^2\le C|y_u(T)-y_v(T)|^2,$$
where $C$ is a constant depending only on $\alpha_T$. Hence, we get from \eqref{ADE10} that
\begin{multline*}
\dfrac{1}{2}|\lambda(t)|^2+\dfrac{\alpha^2}{2}|\nabla \lambda(t)|^2\le C_M\left\{\int_0^T|\nabla y_v(s)-\nabla y_u(s)|^2ds +|y_u(T)-y_v(T)|^2\right.\\
\left.+ \int_t^T|\nabla \lambda(s)|^2ds\right\}
\end{multline*}
since $\|\lambda_v\|_{W^{1,2}(0,T;V)},\,\|y_u\|_{W^{1,2}(0,T;V)}\le C_M$. In addition, we have
$$\int_0^T\|y_v(s)-y_u(s)\|^2ds +|y_u(T)-y_v(T)|^2\le C\|y_u-y_v\|^2_{W^{1,2}(0,T;V)},$$
\begin{align*}
\|y_u-y_v\|^2_{W^{1,2}(0,T;V)}&=\|G(u)-G(v)\|^2_{W^{1,2}(0,T;V)}\\
&\le \sup_{0\le \rho\le 1}\|G'(u+\rho(v-u))\|_{W^{1,2}(0,T;V)}\|u-v\|_{L^2(0,T;\mathbb{L}^2(\Omega))}\\
&\le C_M\|u-v\|_{L^2(0,T;\mathbb{L}^2(\Omega))}.
\end{align*}
Therefore,
$$\|\lambda(t)\|^2_{\mathbb{H}^1(\Omega)}\le C_M\left\{\|u-v\|_{L^2(0,T;\mathbb{L}^2(\Omega))}^2+\int_t^T\|\lambda(s)\|^2_{\mathbb{H}^1(\Omega)}ds\right\},\;\;\forall t\in [0,T].$$
This implies that
$$\|\lambda(t)\|_{\mathbb{H}^1(\Omega)}\le C\|u-v\|_{L^2(0,T;\mathbb{L}^2(\Omega))}\;\;\forall t\in[0,T],$$
by using the Gronwall inequality. Combining this with \eqref{ADE9.1} we get \eqref{ADE9}.
\end{proof}
\subsection{Convergence of the discrete control problem and error estimates}
Since $J_\sigma$ is a continuous and coercive function on a non-empty convex closed subset of a finite-dimensional space, it is easy to see that problem $(P_\sigma)$ has at least one solution. By the similar arguments as in \cite[Sections 4.3 and 4.4]{Casas2012}, we get the following theorems. %The first one shows the convergence of these discrete solutions to a solution of problem $(P)$; the second one asserts that every strict local minimum of problem $(P)$ can be approximated by local minima of problems $(P_\sigma)$; and the last one gives time-space error estimates for the discretization.

The first theorem shows the convergence of these discrete solutions to a solution of problem $(P)$. The proof of this theorem is exactly that of Theorem 4.13 in \cite{Casas2012}.
\begin{theorem}
\label{THR4.4}
Denote by $\bar{u}_\sigma$ a global solution of problem $(P_\sigma)$. Then the sequence $\{\bar{u}_\sigma\}_\sigma$ is bounded in $L^2(0,T;\mathbb{L}^2(\Omega))$ and there exist subsequences, denoted in the same way, weakly convergent in the space $L^2(0,T;\mathbb{L}^2(\Omega))$. If $\bar{u}\in L^2(0,T;\mathbb{L}^2(\Omega))$ is one of the limit points, i.e. $u_\sigma\rightharpoonup \bar{u}$, then $\bar{u}$ is a solution of problem $(P)$. Moreover, we have
\begin{equation}
\label{CD1}
\lim_{\sigma\to 0}\|\bar{u}-\bar{u}_\sigma\|_{L^2(0,T;\mathbb{L}^2(\Omega_h))}=0\quad \text{ and } \quad \lim_{\sigma\to 0}J_\sigma(\bar{u}_\sigma)=J(\bar{u}).
\end{equation}
\end{theorem}
 
In general, it is not correct to claim that the sequence $\{u_\sigma\}_\sigma$ is bounded in $L^2(0,T;\mathbb{L}^2(\Omega))$, because $\bar{u}_\sigma$ is only defined in $(0,T)\times \Omega_h$. We will prove that $\{u_\sigma\}_\sigma$ is bounded in $L^2(0,T;\mathbb{L}^2(\Omega_h))$ by a constant independent of $\sigma$. Then, we take an arbitrary element $v$ in the space $L^2(0,T;\mathbb{L}^2(\Omega))$ and extend every $\bar{u}_\sigma$ to $(0,T)\times \Omega$ by setting $\bar{u}_\sigma(t,x)=v(t,x)$ for every $(t,x)\in (0,T)\times (\Omega\backslash \Omega_h)$. By \eqref{Ap1}, the sequence $\{\bar{u}_\sigma\}_\sigma$ is bounded in $L^2(0,T;\mathbb{L}^2(\Omega))$ and every weak limit point of a subsequence is a solution of $(P)$, regardless of the choice of $v$.

The next theorem, whose proof is the same that of Theorem 4.15 in \cite{Casas2012}, is important from a practical point of view because it states that every strict local minimum of problem $(P)$ can be approximated by local minima
of problems $(P_\sigma)$.
\begin{theorem}
\label{THR4.5}
If $\bar{u}$ is a strict local minimum of $(P)$ then there exists a sequence $\{\bar{u}_\sigma\}_\sigma$ of local minima of problems $(P_\sigma)$ such that \eqref{CD1} holds.
\end{theorem}
We now denote by $\bar{u}$ a locally optimal control of the problem $(P)$, and for every $\sigma$, $\bar{u}_\sigma$ denotes a local solution of $(P_\sigma)$ such that $\|\bar{u}-\bar{u}_\sigma\|_{L^2(0,T;\mathbb{L}^2(\Omega_h))}\to 0$ (see Theorems \ref{THR4.4} and \ref{THR4.5}). Each element $u\in U_\sigma$ is extended to $(0,T)\times\Omega$ by setting $u(t,x)=\bar{u}(t,x)$ for $(t,x)\in (0,T)\times (\Omega\backslash \Omega_h)$. We will also denote by $\bar{y}$ and $\bar{\lambda}$ the state and adjoint state associated to $\bar{u}$, and by $\bar{y}_\sigma$ and $\bar{\lambda}_\sigma$ the discrete state and adjoint state associated to $\bar{u}_\sigma$. 

We are ready to give space-time error estimates for the discretization.
\begin{theorem}
Suppose that \eqref{DGE12} holds. Then there exists a constant $C>0$ independent of $\sigma$ such that
\begin{align}
&\|\bar{u}-\bar{u}_\sigma\|_{L^2(0,T;\mathbb{L}^2(\Omega_h))} \le C(h+\tau+\sqrt{\tau}), \label{MEST1}\\
&\|\bar{y}-\bar{y}_\sigma\|_{L^\infty(0,T;\mathbb{H}^1(\Omega))}\le C(h+\tau+\sqrt{\tau}), \label{MEST2}\\
&\|\bar{\lambda}-\bar{\lambda}_\sigma\|_{L^\infty(0,T;\mathbb{H}^1(\Omega))}\le C(h+\tau+\sqrt{\tau}). \label{MEST3}
\end{align}
\end{theorem}
\begin{proof} The estimates \eqref{MEST2} and \eqref{MEST3} are an immediate consequence of \eqref{MEST1}, \eqref{DS15.2},
and \eqref{ADE9}. We only have to prove \eqref{MEST1}. To this end, we proceed by contradiction and assume that it is false. This implies that
$$
\limsup_{\sigma \to 0} \dfrac{1}{h}\|\bar{u}-\bar{u}_\sigma\|_{L^2(0,T;\mathbb{L}^2(\Omega_h))}=+\infty; 
$$
therefore, there exists a sequence of $\sigma$ such that
\begin{equation*}
\lim_{\sigma \to 0} \dfrac{1}{h}\|\bar{u}-\bar{u}_\sigma\|_{L^2(0,T;\mathbb{L}^2(\Omega_h))}=+\infty.
\end{equation*}
Now, arguing exactly as in \cite[Section 4.4]{Casas2012} and using the second-order optimality conditions, we will obtain a contradiction for this sequence. 
\end{proof}
\begin{remark} {\rm 
The error order in the two last estimates in the above theorem looks a bit different from those (of order $O(h)$) in \cite{Casas2012}. The reason is that we do not require the technical condition $\tau\le Ch^2$ as in \cite{Casas2012}, so the error should contain both $\tau$ and $h$. It is obvious that if $\tau\le Ch^2$ then these orders of errors are the same.}
\end{remark}

\vskip 0.5cm

\noindent{\bf Acknowledgements.}  This research is funded by Vietnam National Foundation for Science and Technology Development (NAFOSTED) under grant number 101.02-2018.303.
%==========================================================================================

\end{document}